\RequirePackage{fix-cm}
\documentclass[11pt, a4paper, oneside, DIV11, final, pagebackref]{amsart}
\usepackage[notref,notcite]{showkeys}
\usepackage[utf8]{inputenc}
\usepackage[T1]{fontenc}
\usepackage{amssymb}
\usepackage[stretch=10]{microtype}
\usepackage{mathtools}
\usepackage[DIV11, headinclude=true]{typearea}
\usepackage{xspace}
\usepackage[hidelinks, linktoc=all]{hyperref}

\newcommand{\calB}{\mathcal{B}}
\newcommand{\calF}{\mathcal{F}}
\newcommand{\calG}{\mathcal{G}}
\newcommand{\calN}{\mathcal{N}}
\newcommand{\calK}{\mathcal{K}}
\newcommand{\calX}{\mathcal{X}}
\newcommand{\given}{\mid}

\newcommand{\p}{{\mathbb{P}}}
\newcommand{\Z}{{\mathbb{Z}}}

\numberwithin{equation}{section}

\newcommand{\R}{\mathbb{R}}
\newcommand{\N}{\mathbb{N}}

\newcommand{\E}{\mathbb{E}}

\newcommand{\dd}{\mathop{}\!\mathrm{d}}

\newcommand{\F}{{\mathcal{F}}}

\DeclareMathOperator{\cov}{\mathrm{cov}}

\DeclarePairedDelimiter{\abs}{\lvert}{\rvert}
\DeclarePairedDelimiter{\accol}{\{}{\}}
\DeclarePairedDelimiter{\pare}{(}{)}
\DeclarePairedDelimiter{\bracke}{[}{]}
\DeclarePairedDelimiter{\norm}{\lVert}{\rVert}

\newcommand{\st}{\,:\,}

\newcommand{\pn}{t}
\newcommand{\qn}{u}
\newcommand{\kn}{n_t}
\newcommand{\nnu}{n_u}
\newcommand{\yn}{y}

\theoremstyle{plain}

\newtheorem{theorem}{Theorem}[section]

\newtheorem{lemma}[theorem]{Lemma}

\newtheorem{prop}[theorem]{Proposition}
\newtheorem{lma}[theorem]{Lemma}
\newtheorem{Definition}[theorem]{Definition}
\newtheorem{Corollary}[theorem]{Corollary}
\newtheorem{proposition}[theorem]{Proposition}

\theoremstyle{definition}

\newtheorem{remark}[theorem]{Remark}
\allowdisplaybreaks[3]

\def\Hun{$\mathbf{H}_1(p)$\xspace}
\def\Hdeux{$\mathbf{H}_2$\xspace}
\def\H{$\mathbf{H}(p)$\xspace}
\def\un{\mathbf{1}}

\title[Moderate deviations for  slowly mixing Markov chains]{Large and moderate deviations for bounded functions of slowly mixing Markov chains}

\date{\today}

\author{J. Dedecker}
\author{S. Gouëzel}
\author{F. Merlevède}

\address{Université Paris Descartes, Sorbonne Paris Cité, Laboratoire MAP5
and CNRS UMR 8145,\newline 45 rue des Saints Pères, \newline 75270
Paris Cedex 06, France.} \email{jerome.dedecker@parisdescartes.fr}

\address{Université de Nantes,
Laboratoire Jean Leray, CNRS UMR 6629, \newline 2 rue de la Houssinière,\newline 44322 Nantes,  France.} \email{sebastien.gouezel@univ-nantes.fr}
\urladdr{\url{http://www.math.sciences.univ-nantes.fr/~gouezel/}}

\address{Université Paris-Est, LAMA
and CNRS UMR 8050, \newline Cité Descartes - 5 boulevard Descartes, \newline
Champs-sur-Marne \newline
77454 Marne-la-Valée Cedex 2} \email{Florence.Merlevede@u-pem.fr}
\urladdr{\url{http://perso-math.univ-mlv.fr/users/merlevede.florence/}}

\begin{document}

\begin{abstract}
We consider Markov chains which are polynomially mixing, in a weak
sense expressed in terms of the space of functions on which the mixing
speed is controlled. In this context, we prove polynomial large and moderate 
deviations inequalities. These inequalities can be applied in various natural
situations coming from probability theory or dynamical systems. Finally, we
discuss examples from these various settings showing that our inequalities
are sharp.
\end{abstract}

\maketitle

\section{Introduction and results}

For stationary $\alpha$-mixing sequences in the sense of Rosenblatt
(see~\cite{rosenblatt_mixing}) a Fuk-Nagaev type inequality has been proved
by Rio (see Theorem 6.2 in~\cite{rio_book}). This deviation inequality is
very powerful and gives for instance sharp upper bounds for the deviation of
partial sums
 when the strong mixing coefficients decrease at a polynomial rate.  In particular for a bounded observable  $f$ of a strictly stationary
 Markov chain $(Y_{i})_{i \in \Z}$ with strong mixing coefficients of order $O(n^{1-p})$ for $p \geq 2$, Rio's inequality gives:
 for any
$x>0$ and any $r \geq 1$,
\begin{equation} \label{rio-FNine}
\p\pare[\Big]{\max_{1 \leq k \leq n} \abs[\Big]{\sum_{i=1}^k (f(Y_i) - \pi(f) )} \geq x} \leq
C \Big \{\frac{n}{x^p} + \frac{n^{r/2}}{x^r} +
 \frac{(n \log n)^{r/2}}{x^r} \un_{p=2} \Big \} \, ,
\end{equation}
where $C$ depends on $\norm{f}_{\infty}$, on $p$ and on $r$.

However, many stationary processes are not strong mixing in the sense of
Rosenblatt. This is the case, for instance,  of the iterates of an ergodic
measure-preserving transformation.  In the recent paper
~\cite{Dedecker_Merlevede_deviation_alpha}, the authors proved that, using a
weaker version of the $\alpha$--mixing coefficients, it is still possible to
get the same upper bound as~\eqref{rio-FNine} but for bounded variation
observables and with the restriction $r \in (2 (p-1), 2p)$. This last
restriction does not affect the asymptotic behavior of the probability of
large deviations (that is when $x=ny$ in~\eqref{rio-FNine} with $y$ fixed)
but gives a restriction for the moderate deviation behavior.

The aim of this paper is to obtain upper bounds of the type~\eqref{rio-FNine}
for stationary Markov chains,  when the mixing property of the chain is
defined through  a subclass of bounded observables ${\mathcal B}$, but
without restriction on $r$. In that case, the deviation inequality (see our
Theorem~\ref{ModDevMC}) will be valid for any observable $f \in {\mathcal
B}$. Maybe the same kind of inequalities can be proved in a more general (but
non $\alpha$-mixing) context than the Markovian setting, but the proof we
give here uses the Markovian property in a crucial way.

\smallskip

Let us now present more precisely the assumptions on the Markov chains and
the main results of the paper.

\smallskip

Let $(Y_i)_{i \in {\mathbb Z}}$ be a homogeneous Markov chain on a state
space $\calX$, with transition operator $K$, admitting a stationary
probability measure $\pi$. Let $\norm{\cdot}$ be a norm on a vector space
$\calB$ of functions from $\calX$ to $\R$. We always require that the
constant function equal to $1$ belongs to $\calB$. This norm will be used to
express mixing conditions on the Markov chain.

We will need this norm to behave well with respect to products, and to be
controlled by the sup norm, as expressed in the next definition.
\begin{Definition}
We say that $\norm{\cdot}$ is a Banach algebra norm on bounded functions if,
for all $f$ and $g$ in $\calB$, one has $\norm{f}_\infty \leq \norm{f}$ and
$\norm{fg} \leq \norm{f}\norm{g}$.
\end{Definition}
\begin{remark}
If a norm $\norm{\cdot}$ satisfies $\norm{f}_\infty \leq C\norm{f}$ and
$\norm{fg} \leq C\norm{f}\norm{g}$ for some constant $C$, then it is
equivalent to a Banach algebra norm on bounded functions, namely
$\norm{f}'=C\norm{f}$.
\end{remark}

The main mixing condition we require is that the iterates of functions in
$\calB$ under the Markov chain converge polynomially to their average. This
is expressed in terms of the following two conditions.
\begin{Definition}
Let $p>1$. We say that the condition \Hun is satisfied if there exists a
positive constant $C_1$ such that, for any function $f \in \calB$ and any
$n\geq 1$,
\begin{equation*}
\tag*{\Hun} \pi \pare[\big]{\abs{K^n(f) - \pi (f)}} \leq \frac{C_1\norm{f}}{n^{p-1}} \, .
\end{equation*}
We say that the condition \Hdeux is satisfied if the space $\calB$ is
invariant under $K$, i.e., there exists a positive constant $C_2$ such that,
for any function $f$ in ${\mathcal B}$,
\begin{equation*}
\tag*{\Hdeux} \norm{K^n (f)} \leq C_2 \norm{f}\, .
\end{equation*}
When both conditions are satisfied, we say that the chain converges
polynomially to equilibrium for the norm $\norm{\cdot}$ with exponent $p$,
and we denote this condition by \H.
\end{Definition}

Heuristically, partial sums of bounded functions of such a polynomially
mixing chain behave like sums of independent random variables with a weak
moment of order $p$. Indeed, if one considers a Harris recurrent Markov chain
for which the excursion time away from an atom has a weak moment of order
$p$, then the successive excursions are independent and have a weak moment of
order $p$, and the mixing rate  behaves like in the definition above. Hence,
one expects that one should prove, under \H, results that are similar to
results for sums of i.i.d.~ random variables with a  weak moment of order
$p$.

In particular, let us consider the question of moderate deviations bounds
\[
\p\pare[\Big]{\max_{1 \leq k \leq n} \abs[\Big]{\sum_{i=1}^k (f(Y_i) - \pi(f) )} \geq x n^{\alpha}} \, ,
\]
where $f$ belongs to ${\mathcal B}$, $x>0$. In analogy with the i.i.d.~ case,
one expects that, if $p\geq 2$, then for any $\alpha \in (1/2,1]$ there
should exist positive constants $C$ depending only on $p$ and on $\norm{f}$,
and $v(x)$ depending only on $x$, such that
\begin{equation}
\label{butnornesup}
\limsup_{n \rightarrow \infty} n^{\alpha p -1} \p\pare[\Big]{\max_{1 \leq k \leq n} \abs[\Big]{\sum_{i=1}^k (f(Y_i) - \pi(f) )}
\geq x n^{\alpha}} \leq C v(x) \, .
\end{equation}

Our main result ensures that this estimate indeed holds, with bounds that are
very similar to the case of the sum of i.i.d.~random variables. We also deal
with the case $p<2$, obtaining similar estimates.

\begin{theorem} \label{ModDevMC}
Let $(Y_i)_{i \in {\mathbb Z}}$ be a stationary Markov chain with state space
$\calX$, transition operator $K$ and stationary measure $\pi$. Assume that
there exists $p>1$ such that \Hun holds, for a Banach algebra norm on bounded
functions.
\begin{enumerate}
\item[(1)] If $p>2$ and we assume in addition that \Hdeux is satisfied
    then, for any $f \in {\mathcal B}$ and any $x>0$,
\begin{equation} \label{eq:main_p>2}
  \p\pare[\Big]{\max_{1 \leq k \leq n} \abs[\Big]{\sum_{i=1}^k (f(Y_i) - \pi(f) )} \geq x}
  \leq \kappa n x^{-p} + \kappa \exp(-\kappa^{-1} x^2/n)\, ,
\end{equation}
where $\kappa$ is a positive constant depending only on $p$, $\norm{f}$,
$C_1$, $C_2$.
\item[(2)] If $p=2$ and we assume in addition that \Hdeux is satisfied
    then, for any $f \in {\mathcal B}$, any $x>0$ and any $r \in (2,4)$,
\begin{equation} \label{eq:main_p=2}
  \p\pare[\Big]{\max_{1 \leq k \leq n} \abs[\Big]{\sum_{i=1}^k (f(Y_i) - \pi(f) )} \geq x}
  \leq \kappa n x^{-2} + \kappa (n \log n)^{r/2} x^{-r} \, ,
\end{equation}
where $\kappa$ is a positive constant depending only on $\norm{f}$, $C_1$,
$C_2$ and $r$.
\item[(3)] If $1 < p < 2$ then, for any $f \in {\mathcal B}$ and any $x
    > 0$,
\begin{equation} \label{eq:main_p<2}
  \p\pare[\Big]{\max_{1 \leq k \leq n} \abs[\Big]{\sum_{i=1}^k (f(Y_i) - \pi(f) )} \geq x}
  \leq \kappa n x^{-p} \, ,
\end{equation}
where $\kappa$ is a positive constant depending only on $p$, $\norm{f}$ and
$C_1$.
\end{enumerate}
\end{theorem}
As a consequence of this theorem, we obtain that, if $p>1$,
then~\eqref{butnornesup} holds with $v(x)= x^{-p}$ for any $\alpha >1/2$ such
that $1/p \leq \alpha \leq 1$ provided that \H holds.

\begin{remark}
In~\eqref{eq:main_p>2}, the exponential term $\kappa \exp(-\kappa^{-1}
x^2/n)$ is negligible in the regime $x>n^\alpha$, for any $\alpha > 1/2$.
Hence, the dominating term is $\kappa n x^{-p}$, as expected. However, when
$x$ is of the order of $n^{1/2}$,  then $\kappa n x^{-p}$ tends to $0$, while
the probability on the left of~\eqref{eq:main_p>2} typically does not, thanks
to the central limit theorem. Thus, there has to be a remainder term, given
here in exponential form $\kappa \exp(-\kappa^{-1} x^2/n)$. For any $r>0$,
this is for instance bounded by $C_{\kappa,r} n^r/x^{2r}$.
\end{remark}

\begin{remark}
In~\eqref{eq:main_p=2}, the scaling in $x^2/n\log n$ in the error term is the
right one: in this setting there is sometimes a central limit theorem with
anomalous scaling $\sqrt{n \log n}$ (see for instance~\cite{gouezel_stable}),
meaning that the probability on the left of~\eqref{eq:main_p=2} does not tend
to $0$ when $x$ is of the order of $\sqrt{n \log n}$.
While~\eqref{eq:main_p>2} is completely satisfactory, we expect that the
error term in~\eqref{eq:main_p=2} can be improved, from $(n \log n)^{r/2}
x^{-r}$ with $r\in (2,4)$ to $(n \log n)^{r/2} x^{-r}$ for any $r>2$, or even
to $\exp(-\kappa^{-1} x^2/(n\log n))$. However, we are not able to prove such
a result.
\end{remark}

\medskip

Let us discuss the relevance of the assumption \H in different contexts. Some
possible Banach algebra norms on bounded functions that appear in natural
examples of Markov chains are the following:
\begin{enumerate}
\item $ \norm{f} = \norm{f}_{\infty}\, . $
\item $ \calX= \R \text{ and } \norm{f}$ is the total variation norm of the
    bounded variation function $f$, i.e., the sum of $\norm{f}_{\infty}$
    and the total variation of the measure $\dd f$, i.e., $\norm{f} =
    \norm{f}_{\infty} + \abs{\dd f}$.
\item If $(\calX,d)$ is a metric space, then one can consider the Lipschitz
    norm
\[ \norm{f} = \norm{f}_{\infty} + \norm{f}_{\mathrm{Lip}} \quad\text{where}\quad
   \norm{f}_{\mathrm{Lip}} \coloneqq \sup_{y\neq z \in
    \calX} \frac{\abs{f(y) - f(z)}}{d (y, z)} \, ,
\]
or Hölder norms.
\item $ \calX = \R \text{ and } \norm{f} = \norm{f}_{\infty} + \norm{f
    '}_{L^r(\lambda)}$ for $r \geq 1$, when $f$ is absolutely continuous
    and $f'$ is its almost sure derivative. One can also consider more
    general Sobolev spaces, in dimension $1$ or higher.
\end{enumerate}

\smallskip

Here is a more detailed discussion of some corresponding examples:

\begin{enumerate}
\item When \Hun is satisfied with $\norm{f} = \norm{f}_\infty$, then the
    chain is said to be strong mixing in the sense of Rosenblatt with
    polynomial rate of convergence $n^{p-1}$, and we write in this case
\[
\alpha_n = \sup_{k \geq n} \sup_{\norm{f}_{\infty} \leq 1} \pi \pare[\big]{\abs{K^k(f) - \pi (f)}} \leq \frac{C_1}{n^{p-1}} \, .
\]
Note that for this norm, \Hdeux is trivially satisfied. In this situation,
one can apply the Fuk-Nagaev type inequality~\cite[Theorem 6.1]{rio_book}
(with $q=cx$ for suitably small $c$). If $p\geq 2$, this gives the
inequality~\eqref{rio-FNine}. Hence,~\eqref{butnornesup} follows (although
the error term is worse than in~\eqref{eq:main_p>2}).

\item A lot of Markov chains, even very simple, are known not to be strong
    mixing whereas they satisfy the condition \H for other classes of
    functions. For instance,
\[
  X_n= \sum_{i=0}^\infty \frac{\xi_{n-i}}{2^{i+1}} \, ,
\]
where $(\xi_i)$ is an i.i.d.~sequence of r.v.'s $\sim {\mathcal B}(1/2)$ is
a Markov chain which is not strong mixing. Its invariant measure is the
Lebesgue measure on $[0,1]$ and its transition Markov operator is given by
\[
 K(f)(x)= \frac{1}{2} \pare[\Big]{f\pare[\Big]{\frac x2}+ f\pare[\Big]{\frac {x+1}{2}}} \, .
\]
It can been shown that it satisfies the condition \H for any $p$ when we
consider $\calX = \R $ and the total variation norm $\norm{f} =
\norm{f}_{BV} = \norm{f}_\infty + \abs{\dd f}$.

\smallskip

When \H is satisfied with the total variation norm (i.e., $\calB$ is the
set of functions of bounded variation), one does not have at our disposal a
Fuk-Nagaev type inequality as in the strong mixing case. If $p > 2$, an
application of the deviation inequality
of~\cite[Proposition~5.1]{Dedecker_Merlevede_deviation_alpha} gives that
for any $x > 0$ and any $r \in (2(p-1) , 2p)$,
\[
 \p\pare[\Big]{\max_{1 \leq k \leq n} \abs[\Big]{\sum_{i=1}^k (f(Y_i) - \pi(f) )} \geq x}
 \leq C \Big \{\frac{n}{x^p} + \frac{n^{r/2}}{x^r}
   + \frac{(n \log n)^{r/2}}{x^r} {\bf 1}_{p=2} \Big \} \, ,
\]
where $C$ is a positive constant depending on $p$, $\norm{f} $, $C_1$ and
$C_2$ but not on $n$ nor on $x$. So, provided that $ p < 1/(1-\alpha)$, one
can take $2p > r > \frac{2 (\alpha p -1)}{2 \alpha -1}$ and it follows
that~\eqref{butnornesup} is satisfied with $v(x) = x^{-p}$. Our main
theorem above shows that this restriction of $\alpha$ is not necessary, by
removing the restriction $r \in (2(p-1), 2p)$ for $p>2$. Our proof follows
the same lines as that in~\cite{Dedecker_Merlevede_deviation_alpha}, but we
can get a better bound by taking advantage of the Markovian setting.

\item Several dynamical examples satisfy the assumption \H when
    $\norm{\cdot}$ is the Lipschitz norm or the Hölder norm. Indeed, there
    is a combinatorial model, called Young tower, that can be used to model
    wide classes of systems and for which the assumption \H is directly
    related to return time estimates to the basis of the tower (this is
    explicitly written, for instance, in~(4.3) of
    \cite{dedecker_merlevede_moment}). We refer the interested readers for
    instance to the introduction of~\cite{gouezel_melbourne}, where
    motivations, examples and definitions are given. Our theorem applies to
    such examples, and improves the previous upper bounds of the literature
    such as~\cite{melbourne_moderate} who obtained, when $\alpha \in (1/2,
    1 )$ and $p \geq 2$, a rate of order $(\ln n)^{1-p} n^{(p-1) (2
    \alpha -1)}$ instead of 
    $n^{\alpha p -1}$ in~\eqref{butnornesup}. Using specific properties of such systems established
    in~\cite{gouezel_melbourne}, we are also able to extend
    Theorem~\ref{ModDevMC} to more general functionals than additive
    functionals, see Theorem~\ref{thm:Young_concentration} in
    Section~\ref{sec:Young_concentration}.
\end{enumerate}

In addition, concerning the exponent of $n$, the bound~\eqref{butnornesup} is
optimal as we shall show in Section~\ref{sec:counterexamples}. More
precisely, we shall give there three different examples for which the
deviation probabilities of Theorem~\ref{ModDevMC} are lower bounded by $c \,
nx^{-p}$ for some $c >0$ and $x$ in an appropriate bandwidth. These three
examples are: a discrete Markov chain on $\N$ for which \H is satisfied for
the sup norm, a class of Young towers with polynomial tails of the return
times for which \H is satisfied for a natural Lipschitz norm, and a Harris
recurrent Markov chain with state space $[0,1]$ for which \H is satisfied for
both the sup norm and the total variation norm.  For each example, the
accurate lower bound is given in
Proposition~\ref{prop:deviation_f_lower},~\ref{prop:deviation_f_lower_Young}
and~\ref{prop:deviation_f_lower_Harris} respectively.

Before this, Section~\ref{sec:Young_concentration} is devoted to the
extension of Theorem~\ref{ModDevMC} to more general functionals in the
specific setting of Young towers. The proof of Theorem~\ref{ModDevMC} is
given in Section~\ref{sec:proof_MainThm}.

\section{Concentration for maps that can be modeled by Young towers}
\label{sec:Young_concentration}

In this section, we extend in the specific setting of Young towers
Theorem~\ref{ModDevMC} to more general functionals. As we will not need
specifics of Young towers, we refer the reader to~\cite{gouezel_melbourne}
for the precise definitions, recalling below only what we need for the
current argument. A Young tower is a dynamical system $T$ preserving a
probability measure $\pi$, on a metric space $Z$, together with a subset
$Z_0$ (the basis of the tower) for which the successive returns to $Z_0$
create some form of decorrelation. Thus, an important feature of the Young
tower is the return time $\tau$ from $Z_0$ to itself, and in particular its
integrability properties.

Starting from any $z\in Z$, there is a canonical way to choose at random a
point among the preimages of $z$ under $T$. This defines a Markov chain $Y_n$
for which $\pi$ is stationary, and which is dual to the dynamics (in the
sense that $Y_0,\dotsc, Y_{n-1}$ is distributed like $T^{n-1}z, \dotsc, z$
when $z$ is picked according to $\pi$). The decorrelation properties of this
Markov chain are related to the return time function $\tau$. Namely, if
$\tau$ has a weak moment of order $p>1$, then the Markov chain satisfies \H
for this $p$.

Proving quantitative estimates for the Markov chain or the dynamics is
equivalent. In this section, we will for simplicity formulate the results for
the dynamics, as the estimates of~\cite{gouezel_melbourne} we will use are
formulated in this context.

The class of functionals for which we will prove moderate deviations is the
class of separately Lipschitz functions: these are the functions
$\calK=\calK(z_0,\dotsc, z_{n-1})$ such that, for all $i \in [0,n-1]$, there
exists a constant $L_i$ (the Lipschitz constant of $\calK$ for the $i$-th
variable) with
\begin{equation*}
\abs{\calK(z_0, \ldots, z_{i-1}, z_i, z_{i+1}, \ldots, z_{n-1})
- \calK(z_0, \ldots, z_{i-1}, z'_i, z_{i+1}, \ldots, z_{n-1})} \leq L_i d(z_i,z_i')
\end{equation*}
for all points $z_0,\dotsc, z_{n-1}, z'_i$. We will write $\E \calK$ for the
average of $\calK$ with respect to the natural measure along trajectories
coming from the dynamics, i.e.,
\begin{equation*}
  \E \calK = \int \calK(z, Tz,\dotsc, T^{n-1}z) \dd\pi(z).
\end{equation*}
The article~\cite{gouezel_melbourne} proves optimal moment estimates for
$\calK-\E \calK$. We can prove moderate deviations for this quantity,
extending in this context the results of Theorem~\ref{ModDevMC} to more
general functionals than additive functionals.

\begin{theorem}
\label{thm:Young_concentration} Consider a Young tower $T:Z\to Z$, for which
the return time $\tau$ to the basis has a weak moment of order $p>1$. Let
$\calK$ be a separately Lipschitz function, with Lipschitz constants $L_i$.
Then
\begin{itemize}
\item If $p>2$, then for all $x>0$ one has
\begin{equation}
\label{eq:concentr_p>2}
  \pi\{z \st \abs{\calK(z,\dotsc, T^{n-1}z) -\E \calK} > x\} \leq \kappa \frac{\sum_{i=0}^{n-1} L_i^p}{x^p}
  + \kappa\exp \pare*{- \kappa^{-1}\frac{x^2}{\sum L_i^2}}\,.
\end{equation}
\item If $p=2$, then for all $x>0$
\begin{multline}
\label{eq:concentr_p=2}
  \pi\{z \st \abs{\calK(z,\dotsc, T^{n-1}z) -\E \calK} > x\} \leq \kappa \frac{\sum_{i=0}^{n-1} L_i^2}{x^2}
  \\
  + \kappa\exp \pare*{- \kappa^{-1}\frac{x^2}{\pare[\big]{\sum L_i^2} \cdot (1+\log(\sum L_i) - \log(\sum L_i^2)^{1/2})}}\,.
\end{multline}
\item If $p<2$, then for all $x>0$ one has
\begin{equation}
\label{eq:concentr_p<2}
  \pi\{z \st \abs{\calK(z,\dotsc, T^{n-1}z) -\E \calK} > x\} \leq \kappa \frac{\sum_{i=0}^{n-1} L_i^p}{x^p}\,.
\end{equation}
\end{itemize}
In all these statements, $\kappa$ is a positive constant that does not depend
on $\calK$ nor $n$.
\end{theorem}

The case $p<2$ is already proved in~\cite[Theorem~1.9]{gouezel_melbourne} and
is included only for completeness. The logarithms in the $p=2$ case are not
surprising: this expression is homogeneous in the $L_i$ (i.e., if one
multiplies all the $L_i$ by a constant then the contribution of the
logarithms does not vary), and it reduces to a multiple of $\log n$ when all
the $L_i$ are equal to $1$. The same expression appears in the moment control
when $p=2$ in~\cite[Theorem~1.9]{gouezel_melbourne}.

To prove this theorem, we use the following deviation inequality
for
martingales,~\cite[Corollary 3']{fuk_martingales} (in which we keep
separately the term corresponding to excess probabilities, as in his
Corollary~3).

\begin{prop}
Let $d_1, \dotsc, d_k $ be a martingale difference sequence with respect to
the non-decreasing  $\sigma$-fields $\calF_0, \ldots, \calF_k$. Let $p\geq
2$. Set $\beta = p/(p+2)$ and $c_p^* = (1-\beta)^2/(2 e^p)$. Then, for all
$x>0$,
\begin{multline*}
  \p\pare[\Big]{\max_{1 \leq j \leq k} \abs[\Big]{\sum_{i=1}^j d_i} \geq x} \leq \sum_{i=1}^k \p (\abs{d_i} \geq \beta x )
  + \frac{2}{\beta^p x^p} \sum_{i=1}^k \norm[\big]{\E(\abs{d_i}^p \un_{\abs{d_i} \leq \beta x} \given \calF_{i-1})}_\infty \\
  + 2 \exp\pare[\Big]{-c_p^* \frac{x^2}{\sum \norm{\E(d_i^2 \given \calF_{i-1})}_\infty}}.
\end{multline*}
\end{prop}

As $\sum_{i=j}^k d_i = \sum_{i=1}^k d_i - \sum_{i=1}^{j-1}d_k$, a similar
result follows for reverse martingale difference sequences, by applying the
previous result to the martingale $d_{k-i}$:
\begin{Corollary} \label{Fukreverse}
Let $d_1, \ldots, d_k $ be a reverse martingale difference sequence w.r.t.\
the non-increasing $\sigma$-fields $\calF_1, \ldots, \calF_{k+1}$ (so $\E(d_i
\given \calF_{i+1}) = 0$ and $d_i$ is $\calF_i$-measurable). Let $p\geq 2$.
Set $\tilde\beta = p/(p+2)$ and $\tilde c_p^* = (1-\beta)^2/(8 e^p)$. Then,
for all $x>0$,
\begin{multline*}
  \p\pare[\Big]{\max_{1 \leq j \leq k} \abs[\Big]{\sum_{i=1}^j d_i} \geq x}
  \leq \sum_{i=1}^k \p (\abs{d_i} \geq \tilde\beta x )
  + \frac{2^{p+1}}{\tilde\beta^p x^p} \sum_1^n \norm[\big]{\E(\abs{d_i}^p \un_{\abs{d_i} \leq \tilde\beta x} \given \calF_{i+1})}_\infty \\
  + 4 \exp\pare[\Big]{-\tilde c_p^* \frac{x^2}{\sum \norm{\E(d_i^2 \given \calF_{i+1})}_\infty}}.
\end{multline*}
\end{Corollary}

We will use the following consequence for reverse martingales having a
conditional weak moment of order $p$, as follows (the same corollary holds as
well for martingales). This is a finer version of~\cite[Corollary
3']{fuk_martingales}, replacing the strong norm there with a weak norm.
\begin{Corollary}
\label{lem:Fuk_weak_norm} Let $d_1, \ldots, d_k $ be a reverse martingale
difference sequence w.r.t.\ the non-increasing $\sigma$-fields $\calF_1,
\ldots, \calF_{k+1}$ (so $\E(d_i \given \calF_{i+1}) = 0$ and $d_i$ is
$\calF_i$-measurable). Let $p\geq 2$. Assume that, for all $i$, $d_i$ has a
conditional weak moment of order $p$ bounded by a constant $M_i$, i.e.,
$\p(\abs{d_i} \geq x \given \calF_{i+1}) \leq M_i^p/x^p$. Then there exists a
constant $C_p$ only depending on $p$ such that, for all $x>0$,
\begin{equation*}
  \p\pare[\Big]{\max_{1 \leq j \leq k} \abs[\Big]{\sum_{i=1}^j d_i} \geq x} \leq \frac{C_p}{x^p} \sum_{i=1}^k M_i^p
  + 4 \exp\pare[\Big]{-C_p^{-1} \frac{x^2}{\sum \norm{\E(d_i^2 \given \calF_{i+1})}_\infty}}.
\end{equation*}
\end{Corollary}
\begin{proof}
We apply Corollary~\ref{Fukreverse} with any $q>p$, for instance $q=p+1$.
Since $\p (\abs{d_i} \geq \tilde\beta x ) \leq M_i^p/(\tilde\beta x)^p$, the
first term in the upper bound of this lemma is bounded as desired. The last
term is also bounded as desired. It remains to handle the terms involving
$x^{-q} \norm[\big]{\E(\abs{d_i}^q \un_{\abs{d_i} \leq \tilde\beta x} \given
\calF_{i+1})}_\infty$. We have
\begin{align*}
  x^{-q} \E(\abs{d_i}^q \un_{\abs{d_i} \leq \tilde\beta x} \given \calF_{i+1})
  & \leq x^{-q} q\int_{u=0}^{\tilde\beta x} u^{q-1} \p( \abs{d_i} \geq u \given \calF_{i+1}) \dd u
  \\& \leq x^{-q} q M_i^p \int_{u=0}^{\tilde\beta x} u^{q-1} u^{-p} \dd u
  = x^{-q} q M_i^p \frac{(\tilde\beta x)^{q-p}}{q-p}
  \leq C M_i^p / x^p.
\end{align*}
Summing these terms over $i$ gives a bound as in the statement of the
corollary.
\end{proof}

We can now start the proof of Theorem~\ref{thm:Young_concentration}. Assume
that $\tau$ has a weak moment of order $p\geq 2$. Starting from a separately
Lipschitz function $\calK$, Chazottes and Gouëzel consider
in~\cite{gouezel_chazottes_concentration} a sequence $(d_k)_{k\geq 0}$ of
reverse martingale differences with respect to the filtration $\calF_k$ of
functions depending only on coordinates $x_k,x_{k+1},\dotsc$, given by
\[
d_k = \E(\calK \given \calF_k) - \E(\calK \given \calF_{k+1}) \, .
\]
Page 869 in~\cite{gouezel_chazottes_concentration}, it is proved that, if
$p>2$, then
\begin{equation*}
\E (d^2_k \given \calF_{k+1} ) \leq \sum_{j \leq k} c_{k-j}^{(0)} L_j^2 \, ,
\end{equation*}
where $c_k^{(0)}$ denotes a generic summable sequence that does not depend on
$\calK$ nor $n$. Therefore,
\begin{equation}
\label{eq:L2_p>2}
  \sum_k \norm{\E (d^2_k \given \calF_{k+1} )}_\infty \leq C \sum L_j^2\,.
\end{equation}
Moreover, if $p=2$,~\cite[Section 4.2]{gouezel_melbourne} shows that
\begin{equation}
\label{eq:L2_p=2}
  \sum_k \norm{\E(d_k^2 \given \calF_{k+1})}_\infty \leq C \pare[\big]{\sum L_i^2}
    \cdot \bracke[\Big]{1+\log\pare[\big]{\sum L_i} - \log\pare[\big]{\sum L_i^2}^{1/2}}\,.
\end{equation}

Now we use the following modification
of~\cite[Lemma~6.2]{gouezel_chazottes_concentration}
\begin{lma} \label{lemma62CG} For all $t > 0$ and all integer $k$,
\[
\p\pare[\big]{\abs{d_k} \geq t \given \calF_{k+1}} \leq C t^{-p} \sum_{j =0}^k L_j^p c_{k-j}^{(0)}
+ C t^{-p} \sup_{h > 0} \pare[\Big]{h^{-1} \sum_{j=k-h +1}^k L_j}^p \, .
\]
\end{lma}
\begin{proof}
We just follow the lines of the proof of Lemma 6.2 in
~\cite{gouezel_chazottes_concentration} up to (6.1). Note that this paper
requires the condition $p>2$ (for the validity of (4.8) there), but Lemma~4.2
in~\cite{gouezel_melbourne} replaces this inequality for $p=2$.

For the first sum we have as in~\cite{gouezel_chazottes_concentration}
\[
\sum_{A_1(z_{\alpha}) \geq t/2} g(z_{\alpha}) \leq C t^{-p} \sum_{j \leq k} L_j^p c_{k-j}^{(0)} \, .
\]
On the other hand, if $h$ denotes the smallest $\ell$ such that
$\sum_{j=k-\ell +1}^k L_j \geq t/2$, then
\begin{equation*}
\sum_{A_2(z_{\alpha}) \geq t/2} g(z_{\alpha}) \leq C \pi(\tau \geq h )
\leq C h^{-p} \leq C t^{-p} \sup_{h > 0} \pare[\Big]{h^{-1} \sum_{j=k-h +1}^k L_j}^p \, .
\qedhere
\end{equation*}
\end{proof}

\begin{proof}[Proof of Theorem~\ref{thm:Young_concentration} when $p\geq 2$]
We apply Lemma~\ref{lem:Fuk_weak_norm} to $d_k$, with
\begin{equation}
\label{eq:defMk}
  M_k^p = C \sum_{j = 0}^k L_j^p c_{k-j}^{(0)}
  + C\sup_{h > 0} \pare[\Big]{h^{-1} \sum_{j=k-h +1}^k L_j}^p
\end{equation}
thanks to Lemma~\ref{lemma62CG}. As $c_k^{(0)}$ is summable, the sum over $k$
of the first term is bounded by $C' \sum L_j^p$. An application of the
Hardy-Littlewood maximal inequality in $\ell^p$ gives
\[
\sum_{k\geq 0} \sup_{h > 0} \pare[\Big]{h^{-1} \sum_{j=k-h +1}^k L_j}^p \leq C \sum_j L_j^p \, .
\]
Hence, the sum over $k$ of the second term in~\eqref{eq:defMk} is also
bounded by $C \sum_j L_j^p$. This shows that the first term in
Corollary~\ref{lem:Fuk_weak_norm} gives rise to a bound $C\sum L_i^p/x^p$.

Finally, the second term in Lemma~\ref{lem:Fuk_weak_norm} gives rise to the
exponential error term in the statement of the theorem, thanks
to~\eqref{eq:L2_p>2} when $p>2$ and to~\eqref{eq:L2_p=2} when $p=2$.
\end{proof}

\begin{remark}
Assume that $p>2$ and for any $i$, $L_i \leq 1$. In this case, integrating
Inequality~\eqref{eq:concentr_p>2} leads to
\[
\norm{\calK - \E \calK}^{2(p-1)}_{\pi, 2(p-1)} \ll n^{p-1} \, .
\]
However, in the case of general $L_i$, we do not recover for this moment the
bound $C (\sum L_i^2)^{p-1}$ proved in~\cite[Theorem~1.9]{gouezel_melbourne}
(consider for instance the case $L_0=1$ and $L_1=\dotsc=L_{n-1}=1/\sqrt{n}$).
This moment bound, combined with Markov inequality, gives
\begin{equation*}
  \pi\{\abs{\calK -\E \calK} > x\} \leq \kappa \frac{\pare[\Big]{\sum L_i^2}^{p-1}}{x^{2p-2}}.
\end{equation*}
For the case where all $L_i$ are of the order of $1$, this bound is worse
than the bound of Theorem~\ref{thm:Young_concentration}. However,
surprisingly, it can be better when the $L_i$ vary a lot, for instance when
$L_0=1$ and $L_1=\dotsc=L_{n-1}=1/\sqrt{n}$, and $x = n^{1/4}$.
\end{remark}

\section{Upper bounds for moderate deviations}

\label{sec:proof_MainThm}

In this section, we prove Theorem~\ref{ModDevMC}. Cases (3) and (2) follow
more or less readily from existing inequalities in the literature, while Case
(1) is really new.

\subsection{Proof of Item (3) in Theorem~\ref{ModDevMC}}

Item (3) follows directly from an application of Proposition~4
in~\cite{Dedecker_Merlevede_LLN_Banach}. Indeed, let $M = \norm{f}_{\infty}$
and
\[
 \gamma(k) = \norm{\E(f (Y_k) \given Y_0) - \pi (f)}_1 \, , \text{ for $k \geq 0$}.
\]
\cite[Proposition~4]{Dedecker_Merlevede_LLN_Banach} together with
stationarity implies that for any integer $q$ in $[1,n]$, and any $x \geq
Mq$,
\[
\p\pare[\Big]{\max_{1 \leq k \leq n} \abs[\Big]{\sum_{i=1}^k (f(Y_i) - \pi(f) )} \geq 4 x} \leq  \frac{4 n M}{x^2} \sum_{i=0}^{q-1} \gamma (i) +
 \frac{2 n}{x q} \sum_{i=q+1}^{2q} \gamma (i) \, .
\]
Note that if $x \geq nM/2$ the bound is trivial since the probability is
equal to zero. It is also trivial if $x \leq M \sqrt{2n}$. Therefore we can
always assume that $ M \leq x \leq n M$ and select $q = [x /M]$. Combined
with the fact that, by \Hun,
\[
\gamma (k) \leq C_1 \norm{f} (k+1)^{1-p} \, ,
\]
this gives
\[
\p\pare[\Big]{\max_{1 \leq k \leq n} \abs[\Big]{\sum_{i=1}^k (f(Y_i) - \pi(f) )} \geq 4 x}
\leq  \Big (  \frac{4 n M^{p-1} C_1 \norm{f} }{2-p  }+ 2^p M^{p-1} C_1 \norm{f}  \Big ) n x^{-p} \, .
\]
This ends the proof of Item (3). \qed

\subsection{A  deviation inequality}

For $r>2$, the Rosenthal inequality for sums of centered i.i.d.~random
variables $Z_i$ is the inequality
\begin{equation}
\label{eq:Rosenthal_iid}
  \E\pare[\Big]{\abs[\Big]{\sum_{i=1}^N Z_i}^r} \ll N \E(\abs{Z_1}^r) + N^{r/2} \E(Z_1^2)^{r/2}\,,
\end{equation}
where the implied constant only depends on $r$. What makes this inequality
extremely useful is that the dominating coefficient $N^{r/2}$ is multiplied
by an $L^2$-norm, which is usually mild to control, while the larger
$L^r$-norm only has a coefficient $N$.

We will use repeatedly a Rosenthal-like inequality for weakly dependent
sequences, due to Merlevède and Peligrad, in the following form which is well
suited for the applications to moderate deviations we have in mind. Note
that, in the following statement, all conditional expectations are of the
form $\E(f(Z_i) \given \calG_0)$ for some $i\geq 2$: this means that suitable
mixing conditions can be used to control such terms. The other two terms are
of Rosenthal-type as in the i.i.d.~case, and can thus be controlled using
minimal knowledge on $Z_1$.
\begin{theorem}
\label{thm:Rosenthal} Let $Z_i$ be a strictly stationary sequence of random
variables, adapted to a filtration $\calG_i$. Write $S_i=\sum_{1}^i Z_k$.
Consider a real number $r>2$. Then, for all $N$ and all $x$,
\begin{multline*}
  \p(\max_{i\leq N} \abs{S_i} \geq x)
  \ll \frac{N}{x} \norm{\E(Z_2 \given \calG_0)}_1
  + \frac{N}{x^r} \E(\abs{Z_1}^r)
  + \frac{N^{r/2}}{x^r} \E(Z_1^2)^{r/2}
\\  + \frac{N}{x^r} \bracke*{\sum_{k=1}^N \frac{1}{k^{1+2\delta/r}}
       \pare[\bigg]{\sum_{i=2}^k \norm{\E(Z_i^2 \given \calG_0) - \E(Z_i^2)}_{r/2}}^\delta}^{r/(2\delta)},
\end{multline*}
where $\delta = \min(1, 1/(r-2)) \in (0,1]$. The implied multiplicative
constant in the inequality only depends on $r$.
\end{theorem}
\begin{proof}
Let $M_i = Z_i - \E(Z_i \given \calG_{i-2})$. Then
\begin{equation}
\label{eq:controle_Si}
  \max_{i\leq N} \abs{S_i} \leq
  \max_{2 \leq 2j \leq N} \abs[\Big]{\sum_{i=1}^j M_{2i}}
  + \max_{1 \leq 2j-1\leq N} \abs[\Big]{\sum_{i=1}^j M_{2i-1}} + \sum_{i=1}^N \abs{\E(Z_i \given \calG_{i-2})}.
\end{equation}
If the maximum of the partial sums $S_i$ is at least $x$, one of these three
terms is at least $x/3$.

First, by Markov inequality and stationarity,
\begin{equation*}
  \p\pare[\Big]{\sum_{i=1}^N \abs{\E(Z_i \given \calG_{i-2})} \geq x/3}
  \leq \frac{3}{x} N \norm{\E(Z_2 \given \calG_0)}_1,
\end{equation*}
giving a term compatible with the statement of the theorem. The two other
terms in~\eqref{eq:controle_Si} are controlled similarly, let us consider for
instance the even indices.

We use first Markov inequality with the exponent $r$, and then the
Rosenthal-like inequality~\cite[Theorem~6]{Merlevede_Peligrad_Rosenthal},
giving
\begin{equation*}
  \p\pare[\Big]{\max_{2 \leq 2j \leq N} \abs[\Big]{\sum_{i=1}^j M_{2i}} \geq \frac{x}{3}}
  \leq C \frac{N}{x^{r}}  \pare*{ \norm{M_1}_r^r + \bracke[\Bigg]{\sum_{k=1}^{N/2} \frac{1}{k^{1+2\delta/r}}
  \norm[\Big]{\E\pare[\Big]{\pare[\big]{\sum_{i=1}^k M_{2i}}^2 \given \calG_0}}_{r/2}^\delta}^{r/(2\delta)}}.
\end{equation*}
Since $\norm{M_1}_r^r \leq 2^r \E(\abs{Z_1}^r)$, the resulting term is
compatible with the statement of the theorem. As $M_{2i}$ is a sequence of
martingale differences with respect to $\calG_{2i}$, we have
\begin{equation*}
  \E\pare[\Big]{\pare[\big]{\sum_{i=1}^k M_{2i}}^2 \given \calG_0}
  =\sum_{i=1}^k \E (M_{2i}^2 \given \calG_0)
  \leq \sum_{i=1}^k \E(Z_{2i}^2 \given \calG_0).
\end{equation*}
Therefore, by stationarity,
\begin{equation*}
  \norm[\Big]{\E\pare[\Big]{\pare[\big]{\sum_{i=1}^k M_{2i}}^2 \given \calG_0}}_{r/2}
  \leq \sum_{i=1}^k \norm{\E(Z_{2i}^2 \given \calG_0) - \E(Z_{2i}^2)}_{r/2} + k \E(Z_1^2).
\end{equation*}
We plug this estimate into the previous equation. The first term gives a
contribution as in the statement of the theorem. On the other hand, the
contribution of the second term $k \E(Z_1^2)$ is
\begin{equation*}
  C \frac{N}{x^{r}} \E(Z_1^2)^{\delta \cdot r/(2\delta)} \cdot
  \bracke[\Bigg]{\sum_{k=1}^{N/2} \frac{1}{k^{1+2\delta/r}} k^\delta}^{r/(2\delta)}
  \leq C \frac{N}{x^r} \E(Z_1^2)^{r/2} C' N^{r/2-1}
  = C'' \frac{N^{r/2}}{x^r} \E(Z_1^2)^{r/2},
\end{equation*}
again one of the terms in the statement of the theorem.
\end{proof}

\begin{remark}
Using different Rosenthal inequalities, one can obtain slightly different
statements. For instance, using the classical Rosenthal inequality of
Burkholder for martingales, one obtains a statement analogous to
Theorem~\ref{thm:Rosenthal}, where the last term in the upper bound is
replaced by
\begin{equation}
\label{eq:naive_Rosenthal}
  \frac{N^{r/2}}{x^r} \norm{\E(Z_2^2 \given \calG_0) - \E(Z_2^2)}_{r/2}^{r/2}.
\end{equation}
This statement uses the decorrelation less strongly than
Theorem~\ref{thm:Rosenthal}: For large $i$, the quantity $\norm{\E(Z_i^2
\given \calG_0) - \E(Z_i^2)}_{r/2}$ is likely much smaller than
$\norm{\E(Z_2^2 \given \calG_0) - \E(Z_2^2)}_{r/2}$. Indeed, it turns out
that, for the application below, Theorem~\ref{thm:Rosenthal} will succeed
while an estimate using~\eqref{eq:naive_Rosenthal} fails (compare for
instance~\eqref{eq:oiupoqsf} below to what would be obtained
using~\eqref{eq:naive_Rosenthal}).
\end{remark}

\subsection{Proof of items (1) and (2) in Theorem~\ref{ModDevMC}}

Item (2) in Theorem~\ref{ModDevMC} follows from an application
of~\cite[Proposition~5.1]{Dedecker_Merlevede_deviation_alpha} (while the
result there applies directly to bounded variation functions, the proof works
in the full generality of Theorem~\ref{ModDevMC}). However, as we shall see
it also follows from our proof as a special case.

The strategy of the proof is to apply the Rosenthal bounds of
Theorem~\ref{thm:Rosenthal} to different parts of $\max_{1\leq k \leq n}
\abs[\Big]{\sum_{i=1}^k (f(Y_i) - \pi(f))}$. To illustrate why this strategy
might work, let us recall a way to prove moderate deviations bounds for sums
of centered i.i.d.~random variables $Z_i$ in $L^p$. Consider an integer $n$
and a real number $x>0$. Let $X_i = Z_i 1_{\abs{Z_i}>n^{1/p}} -\E(Z_i
1_{\abs{Z_i}>n^{1/p}})$ and $X'_i = Z_i-X_i$. Then Rosenthal
inequality~\eqref{eq:Rosenthal_iid} (for sums of independent random
variables) with the exponent $p$ applied to $X_i$ gives $\E(\abs{\sum X_i}^p)
\ll n$, while Rosenthal inequality with some exponent $r>p$ applied to $X'_i$
gives $\E(\abs{\sum X'_i}^r) \ll n^{r/2}$. Combining these two inequalities,
we deduce the moderate deviations bound
\begin{equation*}
  \p\pare[\Big]{\abs[\Big]{\sum_{i=1}^n Z_i} \geq x} \ll \frac{n}{x^p} + \frac{n^{r/2}}{x^r}.
\end{equation*}

We will follow the same strategy in our context: split the sum to be
estimated in two different parts, and apply a Rosenthal inequality (in our
case, Theorem~\ref{thm:Rosenthal}) to each part, with suitable exponents.
Instead of truncating, the splitting will be done by constructing blocks, and
separating a conditional average (which is small in $L^1$, but large in
$L^p$, as $X_i$ above) from the dominating term.

Here is a high level version of the (rather technical) proof to follow.
First, we write $\sum_j f(Y_j) - \pi(f)$ as a sum $\sum B_i$, where $B_i$ is
a sum of $f$ along a block of length $n^{1/p}$. Then, we write $B_i$ as
$(B_i-\E(B_i \given \calG_{i-2}^B)) + \E(B_i \given \calG_{i-2}^B)$, where
$\calG_i^B$ is the natural filtration along which $B_i$ is measurable. Then
$\E(B_i \given \calG_{i-2}^B)$ is small in $L^1$, but possibly large in
$L^\infty$. We control the probability of moderate deviations of $\sum \E(B_i
\given \calG_{i-2}^B)$ by grouping these variables into blocks of size
$\asymp x$, then applying Theorem~\ref{thm:Rosenthal}: all the terms in the
upper bound of this theorem can be controlled, in a straightforward albeit
tedious way, by using the assumption \Hun. Then, to control the probability
of moderate deviations of $\sum (B_i-\E(B_i \given \calG_{i-2}^B))$, we
consider separately the sums along even and odd indices, use that each such
sum is a martingale, and apply an exponential inequality for martingales
(here, Freedman inequality). It follows that, to control the probability of
moderate deviations, it suffices to control the deviations of the conditional
quadratic averages. To handle these, we group them again into blocks of size
$\asymp x$ and apply again Theorem~\ref{thm:Rosenthal}. All the terms in the
upper bound of this theorem can also be controlled directly from \Hun.

Below are the details of the proof.

\begin{proof}[Proof of items (1) and (2) in Theorem~\ref{ModDevMC}]
We will use the following notations throughout the proof. Let $f^{(0)} = f -
\pi(f)$ and $M= \norm{f}_{\infty}$ and $\calF_k = \sigma(Y_i , i \leq k )$
and $\E_k (\cdot) = \E (\cdot \, \given \calF_k)$ and $\E^{(0)}_k (\cdot) =
\E (\cdot \, \given \calF_k)-\E(\cdot)$.

Fix $x>0$ and an integer $n$. It suffices to estimate
\begin{equation} \label{eq:8x}
  \p\pare[\Big]{\max_{1 \leq k \leq n} \abs[\Big]{\sum_{i=1}^k (f(Y_i) - \pi(f)
  )} \geq 8x} \,.
\end{equation}
Indeed, if one proves the theorem for this quantity, then the original result
follows by letting $x' = x/8$, as polynomial bounds involving $x'$ or $x$ are
equivalent. From this point on, we concentrate on bounding~\eqref{eq:8x}.

We first notice that since $\norm[\big]{\max_{1 \leq k \leq n}
\abs[\big]{\sum_{i=1}^k f(Y_i) - \pi(f)}}_{\infty} \leq 2 \norm{f}_{\infty}
n$, we can assume that
\begin{equation} \label{R1otherwisetrivial}
x \leq 4^{-1} \norm{f}_{\infty} n \, ,
\end{equation}
otherwise the probability under consideration equals zero. In addition, we
can also assume that
\begin{equation} \label{R2otherwisetrivial}
x \geq 2 \norm{f}_{\infty} n^{1/p} \, ,
\end{equation}
otherwise what we have to prove is trivial as soon as $\kappa$ is greater or
equal to $(16 \norm{f}_{\infty})^p$. So from now on, we assume the two
restrictions above on $x$.

The strategy to prove the desired inequalities is in two steps. First, we
split the sum into blocks of size
\[
\pn = [ n^{1/p}]
\]
as this is the characteristic size when dealing with mixing bounds of
exponent $p$. Then, we write these blocks as sums of a martingale difference
and a remainder. For each of these two terms, we will prove the desired
estimate on the deviation probability using Theorem~\ref{thm:Rosenthal} with
a suitable exponent $r$. While the different sizes of blocks and the
filtrations we will introduce all depend on $n$, we suppress $n$ from the
notations for brevity.

Let
\[
B_{i} = \sum_{j=(i-1)\pn +1}^{i\pn} f^{(0)} (Y_j) \text{ and } X_{i} = \E (B_{i} \given \calF_{(i-2)\pn} ) \, .
\]
Let $\kn = [n/\pn]$ be the number of size $\pn$ blocks. The following
inequality is then valid:
\[
\max_{1 \leq k \leq n} \abs[\Big]{\sum_{i=1}^k (f(Y_i) - \pi(f)} \leq 2\pn \norm{f}_{\infty} +
 \max_{1 \leq j \leq \kn} \abs[\Big]{\sum_{i=1}^j (B_{i} -X_{i} )} +
 \max_{1 \leq j \leq \kn} \abs[\Big]{\sum_{i=1}^j X_{i}} \, .
\]
Since $2\pn \norm{f}_{\infty} \leq 2 n^{1/p} \norm{f}_{\infty} \leq x $, it
follows that
\begin{equation} \label{firstdecomproba}
\begin{split}
  \p\pare[\Big]{\max_{1 \leq k \leq n} \abs[\Big]{\sum_{i=1}^k (f(Y_i) - \pi(f)} \geq 8 x}
  \leq {} & \p\pare[\Big]{ \max_{1 \leq j \leq \kn} \abs[\Big]{\sum_{i=1}^j (B_{i} -X_{i} )} \geq 3 x} \\
  \\ & + \p\pare[\Big]{\max_{1 \leq j \leq \kn} \abs[\Big]{\sum_{i=1}^j X_{i}} \geq 4 x} \, .
\end{split}
\end{equation}
We will control separately these two terms.

\medskip

\emph{First step: controlling $\p\pare[\Big]{\max_{1 \leq j \leq \kn}
\abs[\Big]{\sum_{i=1}^j X_{i}} \geq 4 x}$.}

Consider some $r\in (2(p-1), 2p)$. We will show that
\begin{equation} \label{firstboundX}
\p\pare[\Big]{\max_{1 \leq j \leq \kn} \abs[\Big]{\sum_{i=1}^j X_{i}} \geq 4 x}
\leq \begin{cases}
 \kappa n x^{-p}& \text{if } p>2\\
 \kappa \pare[\big]{n x^{-2} + (n \log n)^{r/2} x^{-r}} & \text{if } p=2 \, ,\\
 \end{cases}
\end{equation}
where $\kappa$ is a positive constant depending only on $p$, $r$, $\norm{f}$,
$C_1$ and $C_2$ but not on $x$ nor $n$.

With this aim, we first let
\[
\qn = \Big [ \frac{x}{2 \norm{f}_{\infty} n^{1/p}}\Big ] \, ,
\]
and we notice that, by~\eqref{R2otherwisetrivial}, $\qn \geq 1$. We will
regroup the $X_i$ into blocks of length $\qn$, which corresponds to blocks of
size $\pn \qn \asymp x$ for $Y_j$: this is the time scale where the sum over
a block can not exceed $x$. By~\eqref{R1otherwisetrivial},
\begin{equation} \label{minorationkn}
\kn \geq \frac{n}{2\pn} \geq \frac{n}{2 n^{1/p}} \geq \frac{4 x}{2 n^{1/p} \norm{f}_{\infty}} \geq 4 \qn \, .
\end{equation}
Define
\[
U_{i} = \sum_{j=(i-1)\qn +1}^{i\qn} X_{j}\, .
\]
It is measurable with respect to $\calG^U_{i} = \sigma (Y_i \, , \, \ell \leq
i \pn \qn -2\pn)$ thanks to the conditional expectation in the definition of
$X_i$. Since $\norm{X_{i}}_{\infty} \leq 2\norm{f}_{\infty} \pn$, we have
\begin{equation*}
 \max_{1 \leq j \leq \kn} \abs[\Big]{\sum_{i=1}^j X_{i}} \leq 2 \norm{f}_{\infty} \pn \qn +
 \max_{1 \leq j \leq [\kn/\qn]} \abs[\Big]{\sum_{i=1}^j U_i}\, .
\end{equation*}
Since $ 2 \norm{f}_{\infty} \pn \qn \leq x$, it follows that
\begin{equation*}
\p\pare[\Big]{\max_{1 \leq j \leq \kn} \abs[\Big]{\sum_{i=1}^j X_{i}} \geq 4 x}
\leq \p\pare[\Big]{\max_{1 \leq j \leq [\kn/\qn]} \abs[\Big]{\sum_{i=1}^j U_i} \geq 3x},
\end{equation*}
which we will control using Theorem~\ref{thm:Rosenthal} applied to $Z_i=U_i$
and $\calG_i = \calG_i^U$ and $N = [\kn/\qn]$ and the exponent $r$. We should
thus show that all the terms in the upper bound of this theorem are
controlled as in~\eqref{firstboundX}.

\medskip

By using \Hun,
\[
 \norm{\E (U_{2} \given \calG^U_{0} )}_1
 \leq \sum_{i=\qn +1}^{2\qn} \sum_{j=(i-1)\pn +1}^{i\pn}
 \norm{\E (f^{(0)} (Y_j) \given \calF_{0} )}_1 \leq C_1 \norm{f^{(0)}}\frac{\qn \pn}{(\qn\pn)^{p-1}} \, .
\]
Note now that $\qn\pn \geq x (8 \norm{f}_{\infty} )^{-1}$. Therefore,
\begin{equation*}
\frac{[\kn/\qn]}{x} \norm{\E (U_{2} \given \calG^U_{0} )}_1
\leq C_1 \norm{f^{(0)}}\frac{\kn/\qn}{x} \frac{\qn \pn}{(\qn\pn)^{p-1}}
\leq C_1 \norm{f^{(0)}} \pare[\big]{8 \norm{f}_{\infty}}^{p-1} n x^{-p} \, .
\end{equation*}
This handles the first term in the upper bound of
Theorem~\ref{thm:Rosenthal}.

\medskip

To control the term involving $\E(\abs{U_{1}}^r)$, we recall that $U_1$ is a
sum of $\qn$ random variables $X_i$, all bounded in sup norm by
$2\norm{f}_{\infty} \pn$. Any precise inequality for the $r$ norm of a sum
will do here. We use for instance~\cite[Theorem~2.5]{rio_book} with $p=r/2$.
It gives
\begin{equation*}
\E(\abs{U_{1}}^r)
  \leq (\qn r)^{r/2} \pare[\Big]{\sum_{i=1}^{\qn} \norm{X_{1}}_\infty \norm{\E_{0} (X_{i})}_{r/2}}^{r/2}\,.
\end{equation*}
Moreover,
\begin{align*}
  \norm{\E_{0} (X_{i})}_{r/2}&
  \leq \sum_{j=(i-1)t+1}^{it} \norm{\E_{0}(\E_{(i-2)\pn} f^{(0)}(Y_j))}_{r/2}
  = \sum_{j=(i-1)t+1}^{it} \norm{\E_{(i-2)\pn \wedge 0}f^{(0)}(Y_j)}_{r/2}
  \\& \leq \sum_{j=(i-1)t+1}^{it} \bracke[\Big]{\norm{\E_{(i-2)\pn \wedge 0}f^{(0)}(Y_j)}_{\infty}^{r/2-1}
    \cdot \norm{\E_{(i-2)\pn \wedge 0}f^{(0)}(Y_j)}_1}^{1/(r/2)}\,.
\end{align*}
The first sup norm is bounded by $2\norm{f}_\infty \leq 2 \norm{f}$, while
the $L^1$ norm is bounded by $C_1 \norm{f^{(0)}}/(\pn \vee (i-1)\pn)^{p-1}$
thanks to \Hun. Hence, $\norm{\E_{0} (X_{i})}_{r/2} \leq \frac{C
\pn}{\pn^{(p-1)/(r/2)}} \frac{1}{(1 \vee (i-1))^{(p-1)/(r/2)}}$, for some
constant $C$. As $r>2(p-1)$, we get a bound
\begin{equation*}
\E(\abs{U_{1}}^r)
\leq C \qn^{r/2} \pare[\Big]{\sum_{i=1}^\qn \pn\cdot \frac{\pn}{\pn^{(p-1)/(r/2)}} \frac{1}{(1 \vee (i-1))^{(p-1)/(r/2)}}}^{r/2}
\leq C'\frac{ (\pn^2 \qn)^{r/2}}{\pn^{(p-1)}} \cdot \qn^{r/2 - (p-1)}\, .
\end{equation*}
Taking into account that $\qn \pn \leq x/(2 \norm{f}_{\infty})$ and $\kn =
[n/t]$, we derive
\begin{equation*}
\frac{[\kn/\qn]}{x^r} \E(\abs{U_{1}}^r) \leq \kappa n x^{-p}\, .
\end{equation*}
This handles the second term in the upper bound of
Theorem~\ref{thm:Rosenthal}.
\medskip

Let us now control the term involving $\E(U_1^2)$. We have
 \begin{align*}
 \E \pare[\big]{U_1^2}
 & = \sum_{j=1}^{\qn} \sum_{\ell=1}^{\qn} \E (X_{j} X_{\ell}) \\
 & = \sum_{j=1}^{\qn} \sum_{\ell=1}^{\qn} \sum_{k=(j-1)\pn +1}^{j\pn} \sum_{m=(\ell-1)\pn +1}^{\ell \pn}
   \E \bracke[\big]{\E_{(j-2)\pn} (f^{(0)} (Y_k) )\cdot \E_{(\ell-2)\pn} (f^{(0)} (Y_m) )} \,.
\end{align*}
Each such term is equal to $\E \bracke[\big]{\E_{(j-2)\pn \wedge (\ell-2)\pn}
(f^{(0)} (Y_k) )\cdot \E_{(j-2)\pn \wedge (\ell-2)\pn} (f^{(0)} (Y_m) )}$. We
bound one of the factors (corresponding to the minimal $j$ or $\ell$) by
$2\norm{f}_\infty$, and use \Hun to bound the other one in terms of the gap
size, which is at least $(\abs{\ell-j}+1)\pn$. Hence,
\begin{equation*}
 \E \pare[\big]{U_1^{2}}
 \leq 2 \norm{f}_\infty  \sum_{j=1}^{\qn} \sum_{\ell=1}^{\qn} \pn^2 \frac{C_1 \norm{f}}{\pn^{p-1} (\abs{j-\ell}+1)^{p-1}}
 \leq \kappa \qn \pn^2 \times \frac{1}{\pn^{p-1}} \pare[\Big]{1 + (\log n) {\bf 1}_{p=2}}
\end{equation*}
as $p \geq 2$, where $\kappa$ is a positive constant. This yields
 \begin{multline*}
 [n_t/u]^{r/2} \E(U_{1}^2)^{r/2} \leq \kappa^{r/2} \pare[\Big]{\frac{n}{\pn \qn} \qn \pn^2 \times \frac{1}{\pn^{p-1}}}^{r/2} \pare[\Big]{1 + (\log n)^{r/2} {\bf 1}_{p=2}} \\
 \leq \kappa^{r/2} \pare[\Big]{n \pn^2 \times \frac{1}{\pn^{p}}}^{r/2} \pare[\Big]{1 + (\log n)^{r/2} {\bf 1}_{p=2}} \leq 2^{rp/2}\kappa^{r/2} n^{r/p} \pare[\Big]{1 + (\log n)^{r/2} {\bf 1}_{p=2}} \, .
\end{multline*}
Using~\eqref{R2otherwisetrivial}, we note that as $r \geq p$ we have $
x^{r-p} \geq 2^{r-p} \norm{f}_{\infty}^{r-p} n^{r/p -1}$. Hence,
\begin{equation*}
\frac{[n_t/u]^{r/2}}{x^{r}} \E(U_{1}^2)^{r/2} \leq
2^{rp/2}\kappa^{r/2} \pare[\Big]{2^{p-r} \norm{f}_{\infty}^{p-r}
\frac{n}{x^p} + \frac{(n \log n)^{r/2}}{x^r} {\bf 1}_{p=2}} \, .
\end{equation*}
This handles the third term in the upper bound of
Theorem~\ref{thm:Rosenthal}.
\medskip

We analyze now the last term in the upper bound of
Theorem~\ref{thm:Rosenthal}. With this aim, we notice that $\calG^U_{0} =
\calF_{-2\pn} $. Therefore, for any $i \geq 2$,
\begin{align*}
 \norm[\Big]{\E \pare[\big]{U^2_{i} \given \calG^U_{0}} - \E \pare[\big]{U^2_{i}}}_{r/2}
 &\leq \sum_{j=(i-1)\qn +1}^{i\qn} \sum_{\ell=(i-1)\qn +1}^{i\qn} \norm[\Big]{\E \pare[\big]{X_{j} X_{\ell} \given \calF_{-2\pn}} - \E \pare[\big]{X_{j} X_{\ell}}}_{r/2}
 \\&\leq 2\sum_{j=(i-1)\qn +3}^{i\qn +2} \sum_{\ell = j}^{i\qn +2} \norm[\Big]{\E^{(0)}_0 \pare[\big]{X_{j} X_{\ell}}}_{r/2}\, .
\end{align*}
Fix $j \in [(i-1)\qn+3, i\qn +2]$ and $\ell \in [j, i\qn +2]$. Then $X_j
X_\ell$ is a sum of $t^2$ terms of the form $\E_{(j-2)\pn} (f^{(0)}
(Y_k))\cdot \E_{(\ell-2)\pn} (f^{(0)} (Y_m))$ for $k\in [(j-1)\pn +1, j\pn]$
and $m \in [(\ell-1)\pn +1, \ell\pn]$. For each such term, writing
$k'=k-(j-2)\pn$ and $m' = m-(j-2)\pn$, we have
\begin{multline*}
  \norm[\big]{\E^{(0)}_0 \bracke[\big]{\E_{(j-2)\pn} (f^{(0)}(Y_k)) \cdot \E_{(\ell-2)\pn} (f^{(0)} (Y_m))}}_{r/2}
  \\
  = \norm*{K^{(j-2)\pn} \bracke[\big]{(K^{k'} f^{(0)}) \cdot (K^{m'} f^{(0)} )}
    - \pi \bracke[\big]{(K^{k'} f^{(0)}) \cdot (K^{m'}f^{(0)})}}_{\pi, r/2} \, .
\end{multline*}
Therefore
\begin{multline*}
\norm[\big]{\E^{(0)}_0 \bracke[\big]{\E_{(j-2)\pn} (f^{(0)}(Y_k)) \cdot \E_{(\ell-2)\pn} (f^{(0)} (Y_m))}}_{r/2}^{r/2} \\
\leq (8\norm{f}_\infty^2)^{r/2-1} \norm*{K^{(j-2)\pn} \bracke[\big]{(K^{k'} f^{(0)}) \cdot (K^{m'} f^{(0)} )}
    - \pi \bracke[\big]{(K^{k'} f^{(0)}) \cdot (K^{m'}f^{(0)})}}_{\pi, 1} \, .
\end{multline*}
Both functions $K^{k'} (f^{(0)} )$ and $K^{m'}(f^{(0)})$ belong to $\calB$,
with a norm bounded by $C_2 \norm{f^{(0)}}$ thanks to the condition \Hdeux.
As $\norm{\cdot}$ is a Banach algebra norm, their product also belongs to
$\calB$. Applying the condition \Hun to this product, we deduce that
\[
  \norm*{K^{(j-2)\pn} \bracke[\big]{(K^{k'} f^{(0)}) \cdot (K^{m'} f^{(0)} )}
    - \pi \bracke[\big]{(K^{k'} f^{(0)}) \cdot (K^{m'}f^{(0)})}}_{\pi, 1}
    \leq \frac{C_3}{((j-2)\pn)^{p-1}}\, ,
\]
for some constant $C_3$. Combining these inequalities yields
\begin{equation*}
  \norm[\Big]{\E^{(0)}_0 \pare[\big]{X_{j} X_{\ell}}}_{r/2}
  \leq \frac{t^2 C_4}{((j-2)\pn)^{(p-1)/(r/2)}}\, .
\end{equation*}
Therefore, we get that for any $i \geq 2$,
\[
 \norm[\Big]{\E \pare[\big]{U^2_{i} \given \calG_{0}} - \E \pare[\big]{U^2_{i}}}_{r/2}
 \leq C_5 \frac{(\qn \pn)^2}{(i\qn\pn)^{2(p-1)/r}} \, .
\]
As $r > 2(p-1)$, this implies that
\[
 \sum_{i=2}^k \norm[\Big]{\E \pare[\big]{U^2_{i} \given \calG_{0}} - \E \pare[\big]{U^2_{i}}}_{r/2}
 \leq C_6 \frac{(\qn \pn)^2}{(\qn\pn)^{2(p-1)/r}} k^{1 - 2(p-1)/r} \, .
\]
Hence
\begin{multline*}
 [\kn/\qn] \bracke*{\sum_{k=1}^{[\kn /\qn]} \frac{1}{k^{1 +2 \delta/r}}
   \pare[\Big]{\sum_{i=2}^k \norm[\Big]{\E \pare[\big]{U^2_{i} \given \calG^U_{0}} - \E \pare[\big]{U^2_{i}}}_{r/2}}^{\delta}}^{r/(2 \delta)} \\
 \leq C_6^{r/2} \frac{(\qn \pn)^{r}}{(\qn\pn)^{p-1}} \frac{\kn}{\qn}
   \bracke*{\sum_{k=1}^{[\kn /\qn]} \frac{1}{k^{1 +2 \delta/r}} \pare[\Big]{k^{1 - 2(p-1)/r}}^{\delta}}^{r/(2 \delta)} \, .
\end{multline*}
As $r < 2p$, the sum over $k$ is uniformly bounded, independently of $n$ or
$x$. Taking into account that $\qn \pn \leq x/(2 \norm{f}_{\infty})$, we get
that there exists a positive constant $\kappa$ such that
\begin{equation}
\label{eq:oiupoqsf}
  \frac{[\kn/\qn]}{x^r} \bracke*{\sum_{k=1}^{[\kn /\qn]} \frac{1}{k^{1 +2 \delta/r}}
    \pare[\Big]{\sum_{i=2}^k \norm[\big]{\E \pare[\big]{U^2_{i} \given \calG^U_{0}} - \E \pare{U^2_{i}}}_{r/2}}^{\delta}}^{r/(2 \delta)}
  \leq \kappa n x^{-p} \, .
\end{equation}
This handles the last term in the upper bound of Theorem~\ref{thm:Rosenthal}.
Altogether, this proves~\eqref{firstboundX} and concludes the proof of the
first step.

\bigskip

\emph{Second step: controlling $\p\pare[\Big]{\max_{1 \leq j \leq \kn}
\abs[\Big]{\sum_{i=1}^j B_i-X_{i}} \geq 4 x}$.}

We will prove
\begin{equation}
\label{firstboundBmoinsX}
\p\pare[\Big]{\max_{1 \leq k \leq \kn}
\abs[\Big]{\sum_{i=1}^k (B_{i} - X_{i} )} \geq 3 x} \leq
\begin{cases}
 \kappa n x^{-p} + \kappa \exp(-\kappa^{-1} x^2/n)& \text{if } p>2\\
 \kappa n x^{-2} + \kappa \exp(-\kappa^{-1} x^2/(n\log n))& \text{if } p=2 \, ,\\
 \end{cases}
\end{equation}
where $\kappa$ is a positive constant depending only on $p$, $\norm{f}$,
$C_1$ and $C_2$ but not on $x$ nor $n$. Starting
from~\eqref{firstdecomproba}, this upper bound combined
with~\eqref{firstboundX} will end the proof of Items 1 and 2 of the theorem.

To prove~\eqref{firstboundBmoinsX}, we start by setting
\[
d_{i} = B_{i} - X_{i} \, \text{ and } \, \calG^B_{i}= \F_{i\pn} \, ,
\]
and we write the following decomposition:
\begin{multline} \label{decBXwithmartingales}
\p\pare[\Big]{\max_{1 \leq k \leq \kn} \abs[\Big]{\sum_{i=1}^k (B_{i} - X_{i} )} \geq 3 x}
\\ \leq \p\pare[\Big]{\max_{1 \leq 2k \leq \kn} \abs[\Big]{\sum_{i=1}^k d_{2i}} \geq 3 x /2} +
\p\pare[\Big]{\max_{1 \leq 2k -1 \leq \kn} \abs[\Big]{\sum_{i=1}^k d_{2i -1}} \geq 3 x /2} \, .
\end{multline}

Note that $(d_{2i})_{i \in {\mathbb Z}}$ (resp. $(d_{2i-1})_{i \in {\mathbb
Z}}$) is a strictly stationary sequence of martingale differences with
respect to the non decreasing filtration $(\calG^B_{2i})_{i \in {\mathbb Z}}$
(resp. $(\calG^B_{2i-1})_{i \in {\mathbb Z}}$). Therefore, since
$\norm{d_{2i}}_{\infty} \leq 2\pn \norm{f}_{\infty} \leq 2 \norm{f}_{\infty}
n^{1/p}$ a.s., by~\cite[Proposition 2.1]{freedman}, for any $\yn>0$,
\begin{multline} \label{AppliFreedman}
\p\pare[\Big]{\max_{2 \leq 2k \leq \kn} \abs[\Big]{\sum_{i=1}^k d_{2i}} \geq 3x/2} \leq 2 \exp \pare[\Big]{\frac{- 9 x^2}{16 \yn}} +
2 \exp \pare[\Big]{\frac{- 9 x}{16 \norm{f}_{\infty} n^{1/p}}} \\
+ \p\pare[\Big]{\sum_{i=1}^{[\kn/2]} \E \pare[\big]{d^2_{2i} \given \calG^B_{2(i-1)}} \geq \yn} \, .
\end{multline}
Note now that
\[
\E \pare[\big]{d^2_{2i} \given \calG^B_{2(i-1)}} \leq \E \pare[\big]{B^2_{2i} \given \calG^B_{2(i-1)}} \, .
\]
Moreover, by stationarity, we infer that
\[
\sum_{i=1}^{[\kn/2]} \E \pare{B^2_{2i}} \leq 2 n \norm{f}_{\infty} \sum_{k =0}^{\pn -1} \norm{\E_0 (f^{(0)} (Y_k)}_1 \, .
\]
Therefore, by \Hun, there exists a positive constant $\kappa$ depending only
on $p$, $\norm{f}$ and $C_1$ such that
\[
\sum_{i=1}^{[\kn/2]} \E \pare{B^2_{2i}} \leq \kappa n (1 + (\log n) {\bf 1}_{p=2}) ) \, .
\]
Selecting
\begin{equation*}
\yn = \begin{cases}
 \max \pare[\Big]{2 \kappa n , 16 x n^{1/p} \norm{f}_{\infty}}& \text{if } p>2\\
 \max \pare[\Big]{2 \kappa n \log n , 16 x (n \log n) ^{1/2} \norm{f}_{\infty}}& \text{if } p=2 \, ,\\
 \end{cases}
\end{equation*}
and starting from~\eqref{AppliFreedman}, we get that, for any $r \geq 1$,
\begin{multline} \label{AppliFreedman2}
\p\pare[\Big]{\max_{2 \leq 2k \leq \kn} \abs[\Big]{\sum_{i=1}^k d_{2i}} \geq 3x/2}
  \leq c \frac{n}{x^p} + c \exp\pare[\Big]{-c' \frac{x^2}{n+ n \log n {\bf 1}_{p=2}}}
\\
+ \p\pare[\Big]{\sum_{i=1}^{[\kn/2]} \pare[\big]{\E \pare[\big]{B^2_{2i} \given \calG^B_{2(i-1)}} - \E \pare{B^2_{2i}}} \geq \yn /2} \, ,
\end{multline}
where $c$ and $c'$ are positive constants.

Let us prove now that
\begin{equation} \label{firstboundBmoinsXB}
 \p\pare[\Big]{\abs[\Big]{\sum_{i=1}^{[\kn/2]} \pare[\Big]{\E \pare[\big]{B^2_{2i} \given \calG^B_{2(i-1)}}
   - \E \pare{B^2_{2i}}}} \geq \yn/2}
 \leq c n x^{-p} \, ,
\end{equation}
where $c$ is a positive constant depending only on $p$, $\norm{f}$, $C_1$ and
$C_2$ but not on $x$ nor $n$. A similar bound will hold for odd indices.
Hence, starting from~\eqref{decBXwithmartingales} and considering the
inequality~\eqref{AppliFreedman2}, this upper bound will lead
to~\eqref{firstboundBmoinsX} and then will end the proof of Items 1 and 2 of
the theorem.

It remains then to prove~\eqref{firstboundBmoinsXB}. With this aim, we do
again blocks of size $\qn$ with as before $\displaystyle \qn = \Big [
\frac{x}{2 \norm{f}_{\infty} n^{1/p}}\Big ] $. Let
\[
W_{i} = \E \pare[\big]{B^2_{2i} \given \calG^B_{2(i-1)}} - \E \pare{B^2_{2i}} \,,\quad  V_{i} = \sum_{k=(i-1)\qn +1}^{i\qn} W_{k}
\]
and $ \calG^V_{i} = \calG^B_{2(i\qn-1)} = \calF_{2(i \qn-1)\pn}$. Setting
$\displaystyle \nnu = \Big [ \frac{[\kn/2]}{\qn}\Big ]$ (note that,
by~\eqref{minorationkn}, $\nnu \geq 1$), we have
\begin{equation*}
\abs[\Big]{\sum_{i=1}^{[\kn/2]} \pare[\big]{\E \pare[\big]{B^2_{2i} \given \calG^B_{2(i-1)}} - \E \pare{B^2_{2i}}}} = \abs[\Big]{\sum_{i=1}^{[\kn/2]} W_{i}}
\leq \abs[\Big]{\sum_{i=1}^{\nnu} V_{i}} + 8 \qn \pn^2 \norm{f}_{\infty}^2\, .
\end{equation*}
Note that
\[
 8 \qn \pn^2 \norm{f}_{\infty}^2 \leq 4 x n^{1/p} \norm{f}_{\infty}
 \leq \yn/4\, .
\]
Therefore
\begin{equation*}
 \p\pare[\Big]{\abs[\Big]{\sum_{i=1}^{[\kn/2]} \pare[\Big]{\E \pare[\big]{B^2_{2i} \given \calG^B_{2(i-1)}} - \E \pare{B^2_{2i}}}}\geq \yn/2}
 \leq \p\pare[\Big]{\abs[\Big]{\sum_{i=1}^{\nnu} V_{i}} \geq \yn/4}\, .
\end{equation*}
To prove~\eqref{firstboundBmoinsXB}, it suffices to show that
\begin{equation} \label{firstboundBmoinsXB2}
 \p\pare[\Big]{\abs[\Big]{\sum_{i=1}^{\nnu} V_{i}} \geq \yn/4}\leq c n x^{-p}\, .
\end{equation}
We will show this inequality by applying Theorem~\ref{thm:Rosenthal} to
$Z_i=V_i$ and $\calG_i = \calG^V_{i}$ and $N = \nnu$ and some fixed $r \in
(2p-2, 2p)$. We should thus show that all the terms in the upper bound of
this theorem are controlled as in~\eqref{firstboundBmoinsXB2}.

\medskip

We start with the first term involving $\norm{\E(V_2 \given \calG^V_0)}_1$.
Since $\yn \geq 16 x n^{1/p} \norm{f}_{\infty} $, we have
\begin{align*}
  \frac{\nnu}{\yn} \norm{\E(V_2 \given \calG^V_0)}_1
 &\leq \frac{\nnu}{x n^{1/p} \norm{f}_{\infty}} \sum_{k=\qn +1}^{2\qn} \norm{\E (W_{k} \given \calF_{-2\pn})}_1 \\&
 \leq \frac{\nnu}{x n^{1/p} \norm{f}_{\infty}} \sum_{k=\qn +1}^{2\qn} \norm{\E (B^2_{2k} \given \calF_{-2\pn}) - \E (B^2_{2k} )}_1 \\&
 \leq \frac{2\nnu}{x n^{1/p} \norm{f}_{\infty}} \sum_{k=\qn +1}^{2\qn} \sum_{j=(2k-1)\pn +1}^{2k\pn}
   \sum_{\ell=j}^{2k\pn} \norm{\E^{(0)} (f^{(0)} (Y_j) f^{(0)} (Y_{\ell}) \given \calF_{-2\pn})}_1 \, .
\end{align*}
Using \H, we infer that there exists a positive constant $c$ depending on
$C_1$, $C_2$ and $\norm{f^{(0)}}$, such that this quantity is bounded by
\[
  c \frac{n}{x n^{1/p} \pn \qn \norm{f}_{\infty}} \qn \pn^2 \frac{1}{(\qn \pn)^{p-1}}
  \leq c \frac{n}{x \norm{f}_{\infty}} \frac{1}{(\qn \pn)^{p-1}}
  \leq 8^{p-1} c \norm{f}_{\infty}^{p-2} n x^{-p} \, ,
\]
thanks to the inequality $\qn \pn \geq x (8 \norm{f}_{\infty} )^{-1}$. This
handles the first term in the upper bound of Theorem~\ref{thm:Rosenthal}.

\medskip

We turn to the second term, involving $\E(\abs{V_1}^r)$. By stationarity
and~\cite[Theorem~2.5]{rio_book}, we have
\begin{align*}
  \E(\abs{V_1}^r)^{2/r}
    & = \norm[\Big]{\sum_{k=1}^{\qn} W_{k}}_r^2 \leq \qn r \sum_{k=0}^{\qn-1} \norm{W_{0} \E (W_{k} \given \calF_{-2\pn} )}_{r/2}
  \\& \leq \qn r \norm{W^2_{0}}_{r/2} + \qn r \norm{W_{0}}_{\infty}\sum_{k=1}^{\qn-1} \norm{\E (W_{k} \given \calF_{-2\pn} )}_{r/2}
  \\& \leq \qn r \norm{W^2_{0}}_{r/2} + \qn r \cdot 16 \pn^2 \norm{f}^2_{\infty} \sum_{k=1}^{\qn-1} \norm{\E (W_{k} \given \calF_{-2\pn} )}_{r/2} \, .
\end{align*}
Using \H, we infer that there exists a positive constant $c_4$ depending on
$C_1$, $C_2$, $\norm{f}$ and $r$ such that
\begin{align*}
  \norm{W^2_{0}}_{r/2}
  &= \norm{W_{0}}^2_{r}
  = \norm{\E_0 \pare[\big]{B^2_{2}} - \E (B^2_{2} )}^2_{r} \\&
  \leq \pare[\Big]{\sum_{j=\pn +1}^{2 \pn} \sum_{i=\pn +1}^{2 \pn} \norm{\E^{(0)}_0 (f^{(0)} (Y_j) f^{(0)} (Y_{i}) )}_r}^2
  \\&
  \leq c_4 \pare[\Big]{\pn^2 \frac{1}{\pn^{(p-1)/r}}}^2
  \leq c_4 \frac{\qn}{\qn^{2(p-1)/r}} \cdot \frac{\pn^4}{\pn^{2(p-1)/r}}
  \, ,
\end{align*}
as $r>2(p-1)$. On the other hand, using again \H, we get that there exists a
positive constant $c_5$ such that for any $k \geq 1$,
\begin{align*}
  \norm{\E (W_{k} \given \calF_{-2\pn} )}_{r/2}
  \leq \sum_{j=(2k-1)\pn +1}^{2k \pn} \sum_{i=(2k-1)\pn +1}^{2k \pn} \norm{\E^{(0)}_{-2\pn} (f^{(0)} (Y_j) f^{(0)} (Y_{i})))}_{r/2}
  \leq c_5 \pn^2 \frac{1}{(k \pn)^{2(p-1)/r}} \, .
\end{align*}
The sum of these quantities over $k$ from $1$ to $u-1$ is bounded by $c_6
\frac{t^2}{t^{2(p-1)/r}} \cdot \frac{u}{u^{2(p-1)/r}}$, as $r > 2(p-1)$. We
infer that there exists a positive constant $c_7$ such that
\[
\nnu \E(\abs{V_1}^r)\leq c_7 \nnu (\qn^2 \pn^4)^{r/2} \frac{1}{(\qn \pn)^{p-1}} \leq c_7 n \pn^r (\qn \pn)^{r-p} \, .
\]
Hence, using the fact that $\yn \geq 16 x n^{1/p} \norm{f}_{\infty} $ and
$\qn \pn \leq x (2 \norm{f}_{\infty})^{-1}$ and $\pn \leq n^{1/p}$, we get
that
\begin{equation*}
  \frac{\nnu}{\yn^{r}} \E(\abs{V_1}^r)
  \leq 8^{-r} c_7 (2 \norm{f}_{\infty})^{p-2r}\frac{n}{x^p} \, .
\end{equation*}
This handles the second term in the upper bound of
Theorem~\ref{thm:Rosenthal}.

\medskip

We turn to the third term, involving $\E(V_1^2)$. By stationarity, we have
\[
\E(V_1^2) = \norm[\Big]{\sum_{k=\qn +1}^{2\qn} W_{k}}_2^2 = \qn \norm{W_{1}}^2_2 + 2 \sum_{k=1}^{\qn-1} \sum_{\ell=1}^{\qn-k} \cov \pare[\big]{W_{0} , W_{\ell }} \, .
\]
But, by using \H, we infer that there exists a positive constant $c_1$ such
that
\begin{equation*}
  \norm{W_{1}}_2
  = \norm{\E^{(0)}_0 (B^2_{2} )}_2
  \leq \sum_{j=\pn +1}^{2\pn} \sum_{\ell=\pn +1}^{2\pn} \norm{\E^{(0)}_0 (f^{(0)} (Y_j) f^{(0)} (Y_{\ell}))}_2
  \leq c_1 \frac{\pn^2}{\pn^{(p-1)/2}} \, .
\end{equation*}
On the other hand, using again \H, we get that there exists a positive
constant $c_2$ such that for any $\ell \geq 1$,
\begin{align*}
 \abs[\Big]{\cov \pare[\big]{W_{0} , W_{\ell }}}
 &\leq \norm{B^2_{0}}_{\infty} \norm{\E_{-2\pn} (B^2_{2\ell} ) - \E (B^2_{2\ell} )}_1
 \\&
 \leq (8\pn \norm{f})^2 \sum_{j=(2 \ell-1)\pn +1}^{2 \ell \pn} \sum_{i=(2 \ell-1)\pn +1}^{2 \ell \pn} \norm{\E^{(0)}_0 (f^{(0)} (Y_j) f^{(0)} (Y_{i}))}_1
 \leq c_2 \frac{\pn^4}{(\ell \pn)^{p-1}} \, .
\end{align*}
So, overall, there exists a positive constant $c_3$ such that
\[
  \E(V_1^2) \leq c_3 \qn \frac{\pn^4}{\pn^{p-1}} \pare[\big]{1 + (\log n ) {\bf 1}_{p=2}}\, .
\]
This upper bound implies that
\[
  \pare[\Big]{\nnu \E(V_1^2)}^{r/2} \leq (2^p c_3)^{r/2} n^{2r/p} \pare[\big]{1 + (\log n )^{r/2} {\bf 1}_{p=2}} \, .
\]
Next using the fact that $\yn \geq 16 x n^{1/p} \norm{f}_{\infty} $ if $p>2$
and $\yn \geq 16 x (n \log n) ^{1/2} \norm{f}_{\infty}$ if $p=2$, we get
\[
\frac{\nnu^{r/2}}{\yn^{r}} \E(V_1^2)^{r/2} \leq (2^p c_3)^{r/2} \frac{n^{r/p}}{16^r x^r \norm{f}^r_{\infty}} \, .
\]
By~\eqref{R2otherwisetrivial} and since $r \geq p$, we have $x^{r-p} \geq (2
\norm{f}_{\infty})^{r-p} n^{r/p -1}$. Therefore,
\begin{equation*}
  \frac{\nnu^{r/2}}{\yn^{r}} \E(V_1^2)^{r/2}
  \leq (2^{p-6} c_3)^{r/2} (2 \norm{f}_{\infty})^{p-2r} \frac{n}{x^p} \, .
\end{equation*}
This handles the third term in the upper bound of
Theorem~\ref{thm:Rosenthal}.

\medskip

Finally, we turn to the last term, involving $\norm{\E (V_{i}^2 \given
\calG_{0}^V ) - \E (V_{i}^2 )}_{r/2}$. For any $i \geq 2$, we have
\begin{align*}
  \norm{\E (V_{i}^2 \given \calG_{0}^V ) - &\E (V_{i}^2 )}_{r/2}
  \leq \sum_{\ell=(i-1)\qn +1}^{i\qn} \sum_{m=(i-1)\qn +1}^{i\qn} \norm{\E_{-2\pn}^{(0)} (W_{\ell} W_{m} )}_{r/2} \\
  \\&\leq \sum_{\ell=(i-1)\qn +1}^{i\qn} \sum_{m=(i-1)\qn +1}^{i\qn}
     \norm*{\E_{-2\pn}^{(0)} \bracke[\Big]{\E \pare[\big]{B^2_{2\ell} \given \calG^B_{2(\ell-1)}} \cdot \E \pare[\big]{B^2_{2m} \given \calG^B_{2(m-1)}}}}_{r/2} \\
  \\& \hspace{1cm} + 2 \sum_{\ell=(i-1)\qn +1}^{i\qn} \sum_{m=(i-1)\qn +1}^{i\qn} \E(B^2_{2\ell})\cdot \norm{\E_{-2\pn}^{(0)} (B^2_{2m})}_{r/2} \, ,
\end{align*}
where this expansion is obtained from the definition $W_i = \E
\pare[\big]{B^2_{2i} \given \calG^B_{2(i-1)}} - \E \pare{B^2_{2i}}$ by expanding the
product $W_\ell W_m$, using the fact that $\E_{-2\pn}^{(0)}$ is linear and
vanishes on the constant $\E(B^2_{2\ell}) \cdot \E(B^2_{2m})$.

For any $ m \geq \ell \geq 1$,
\begin{align*}
 \norm[\Big]{&\E_{-2\pn}^{(0)} \bracke[\Big]{\E \pare[\big]{B^2_{2\ell} \given \calG^B_{2(\ell-1)}} \cdot \E \pare[\big]{B^2_{2m} \given \calG^B_{2(m-1)}}}}_{r/2}
 \\&= \norm[\Big]{\E_{-2\pn}^{(0)} \bracke[\Big]{\E \pare[\big]{B^2_{2\ell} \given \calG^B_{2(\ell-1)}} \cdot \E \pare[\big]{B^2_{2m} \given \calG^B_{2(\ell -1)}}}}_{r/2}
\\& \leq \sum_{a,a'=(2\ell-1)\pn +1}^{2\ell\pn} \sum_{b,b'=(2m-1)\pn +1}^{2m \pn}
  \norm[\Big]{\E^{(0)}_{-2\pn} \bracke[\Big]{\E_{2(\ell-1)\pn} \pare[\big]{f^{(0)} (Y_a) f^{(0)} (Y_{a'})} \\ &\hspace{6cm}\times \E_{2(\ell-1)\pn} \pare[\big]{f^{(0)} (Y_b) f^{(0)} (Y_{b'})}}}_{r/2}
\\& \leq 2 \pn^2 \sup_{a', b' \geq 0} \sum_{a=(2\ell+1)\pn +1}^{2(\ell+1)\pn}\sum_{b=(2m+1)\pn +1}^{2(m+1)\pn}
  \norm[\Big]{\E^{(0)}_{0} \bracke[\Big]{\E_{2\ell\pn} \pare[\big]{f^{(0)} (Y_a) f^{(0)} (Y_{a+a'})} \\&\hspace{6cm}\times \E_{2\ell\pn} \pare[\big]{f^{(0)} (Y_b) f^{(0)} (Y_{b+b'})}}}_{r/2} \, ,
\end{align*}
where we have used stationarity. But
\[
 \E_{2\ell\pn} \pare[\big]{f^{(0)} (Y_a) f^{(0)} (Y_{a+a'})} = \E_{2\ell\pn} \pare[\big]{(f^{(0)} K^{a'} f^{(0)} ) (Y_a)} = (K^{a- 2\ell\pn} (f^{(0)} K^{a'} f^{(0)} ) ) (Y_{2\ell \pn}) \, .
\]
Hence,
\begin{multline*}
 \E_{0} \bracke[\Big]{\E_{2\ell\pn} \pare[\big]{f^{(0)} (Y_a) f^{(0)} (Y_{a+a'})} \cdot \E_{2\ell\pn} \pare[\big]{f^{(0)} (Y_b) f^{(0)} (Y_{b+b'})}} \\
 = \pare[\Big]{K^{2\ell\pn} \bracke[\Big]{K^{a- 2\ell\pn} (f^{(0)} K^{a'} f^{(0)} ) \cdot K^{b- 2\ell\pn} (f^{(0)} K^{b'} f^{(0)})}}(Y_{0}) \, .
\end{multline*}
Therefore, thanks to \H, we infer that there exists a positive constant $c_8$
such that for any $ m \geq \ell \geq (i-1)\qn +1$,
\begin{equation*}
 \norm[\Big]{\E_{-2\pn}^{(0)} \bracke[\Big]{\E \pare[\big]{B^2_{2\ell} \given \calG^B_{2(\ell-1)}}\cdot \E \pare[\big]{B^2_{2m} \given \calG^B_{2(m-1)}}}}_{r/2} \\
 \leq c_8 \frac{\pn^4}{((i-1)\pn\qn)^{2(p-1)/r}}\, .
\end{equation*}
On the other hand, using again \H, we infer that there exists a positive
constant $c_9$ such that for any $\ell, m \geq (i-1)\qn +1$,
\[
 \E (B^2_{2\ell}) \cdot \norm{\E_{-2\pn}^{(0)} (B^2_{2m})}_{r/2} \leq c_9 \pn^2 \cdot \frac{\pn^2}{((i-1)\pn\qn)^{2(p-1)/r}} \, .
\]
So, overall, as $r > 2(p-1)$, there exists a positive constant $c_{10}$ such
that
\[
\sum_{i=1}^k \norm{\E (V_{i}^2 \given \calG_{0}^V ) - \E (V_{i}^2 )}_{r/2} \leq c_{10} \frac{\qn^2\pn^4}{(\pn \qn)^{2(p-1)/r}} \frac{k}{k^{2(p-1)/r}} \, .
\]
Therefore, as in addition $r < 2p$, there exists a positive constant $c_{11}$
such that
\[
 \bracke*{\sum_{k=1}^{\nnu} \frac{1}{k^{1 + 2 \delta /r}} \pare[\Big]{\sum_{i=1}^k \norm{\E (V_{i}^2 \given \calG_{0}^V ) - \E (V_{i}^2 )}_{r/2}}^{\delta}}^{r/(2 \delta)}\leq c_{11}
\frac{(\qn\pn^2)^{r}}{(\pn \qn)^{p-1}} \, .
\]
Using, the fact that $\yn \geq 16 x n^{1/p} \norm{f}_{\infty} $, $\qn \pn
\leq x (2 \norm{f}_{\infty})^{-1}$ and $\pn \leq n^{1/p}$, this implies that
\begin{equation*}
  \frac{\nnu}{\yn^r} \bracke*{\sum_{k=1}^{\nnu} \frac{1}{k^{1 + 2 \delta /r}}
    \pare[\Big]{\sum_{i=1}^k \norm{\E (V_{i}^2 \given \calG_{0}^V ) - \E (V_{i}^2)}_{r/2}}^{\delta}}^{r/(2 \delta)}
 \leq 8^{-r} c_{11} (2 \norm{f}_{\infty})^{p-2r} \frac{n}{x^p}
 \, .
\end{equation*}
This handles the last term in the upper bound of Theorem~\ref{thm:Rosenthal}.
Altogether, this proves~\eqref{firstboundBmoinsXB2}. This concludes the
second step, and therefore the proof of Items~1 and~2 of the theorem.
\end{proof}

\section{Lower bounds in moderate deviations: three examples}

\label{sec:counterexamples}

In this section, we exhibit several examples of Markov chains satisfying \H
(for different  norms) for which one can prove a lower bound for
the deviation probability of some particular observables. This shows that the
upper bounds given in Theorem~\ref{ModDevMC} cannot be essentially improved.

\subsection{Discrete Markov chains}

Let $p>1$. We consider a simple renewal type Markov chain on $\N$, jumping
from $0$ to $n>0$ with probability $p_{0,n} \coloneqq 1/(\zeta(p+1)n^{p+1})$
and from $n>0$ to $n-1$ with probability $1$. This Markov chain has an
invariant probability measure $\pi$ given by $\pi\{n\} = \sum_{i\geq n}
d/i^{p+1}$ for $n>0$ and $\pi\{0\} = \pi\{1\}$, where $d>0$ is chosen so that
$\pi$ is of mass $1$.

This Markov chain satisfies \H for the norm $\norm{f} = \norm{f}_\infty$.
Indeed, in this case,
\[
 \pi \pare[\Big]{\sup_{\norm{f}_{\infty} \leq 1} \abs{K^n(f) - \pi (f)}}
   \leq C_1 \sum_{j \geq n} \sum_{k \geq j+1} p_{0,k} \leq C_2 n^{1-p}
\]
(see~\cite{davydov_MC} or Chapter 30 in~\cite{bradley_vol3} for more
details).

Define a function $f$ by $f(n) = \pi\{0\} - \un_{n=0}$. Its average under
$\pi$ vanishes.
\begin{proposition}
\label{prop:deviation_f_lower} Let $(Y_i)_{i \in \N}$ be a stationary Markov
chain with transition kernel described above, for some $p>1$. There exists
$\kappa>0$ such that, for any $n\in \N^*$ and any $x \in [\kappa n^{1/p},
\kappa^{-1} n]$,
\begin{equation*}
  \p\pare*{\sum_{i=0}^{n-1} f(Y_i) \geq x} \geq \kappa^{-1} \frac{n}{x^p}.
\end{equation*}
\end{proposition}

This paragraph is devoted to the proof of this proposition. Since we are
looking for lower bound, it suffices to consider trajectories starting from
$0$. Denote by $\tau_0, \tau_1,\dotsc$ the lengths of the successive
excursions outside of $0$. This is a sequence of i.i.d.~random variables with
a weak moment of order $p$, namely: $\p ( \tau_0 > n \given Y_0=0) = \sum_{i
\geq n} 1/(\zeta(p+1) i^{p+1})$. We first consider the case $p>2$, and
indicate then the modifications to be done when $p=2$ and when $p\in (1,2)$.

First, we study the probability that the lengths of excursions differ much
from their average.
\begin{lemma}
\label{lem:deviation_tau}  Assume $p>2$. There exists $C_1>0$ such that, for
any $n\geq 1$ and any $x\geq n^{1/p}$, one has
\begin{equation*}
  \p\pare[\Big]{\sum_{i=0}^{n-1} \tau_i \geq n \E(\tau) + x} \geq C_1^{-1} \frac{n}{x^p}.
\end{equation*}
\end{lemma}
\begin{proof}
Write $\bar\tau_i = \tau_i - \E(\tau_i)$. There exists $\sigma^2>0$ such that
$\sum_{i=0}^{n-1} \bar \tau_i / \sqrt{n}$ converges to $\calN(0,\sigma^2)$.
It follows that, for $x\in [n^{1/p}, n^{1/2}]$, the left hand side in the
statement of the lemma converges to a quantity which is bounded from below by
$\p(\calN(0,\sigma^2) \geq 1)>0$, while the right hand side is bounded from
above by $C_1^{-1}$. Taking $C_1$ large enough, the conclusion of the lemma
follows in this range of $x$.

Let us now assume $x\geq \sqrt{n}$. For $i<n$, let
\begin{equation*}
  A_i = \{\bar \tau_i \geq 3x \} \cap \accol[\Big]{\sum_{j=0}^{i-1} \bar\tau_j \leq x}
  \cap \accol[\Big]{\sum_{j=i+1}^{n-1} \bar \tau_j \leq x}\,.
\end{equation*}
This decomposition is the intersection of three independent sets. The first
one has probability at least $c/x^p$ as $\tau$ has polynomial tails of order
$p$, while the measure of the other ones is bounded from below thanks to the
central limit theorem for $\bar\tau$, as we assume $x \geq \sqrt{n}$. Hence,
for some constant $c_1$, we obtain
\begin{equation*}
  \p(A_i) \geq c_1/x^p.
\end{equation*}
Moreover, $A_i \cap A_j$ is contained in $\{\bar \tau_i \geq 3x \} \cap
\{\bar \tau_j \geq 3x \}$. By independence, this set has probability at most
$c_2 / x^{2p}$ for some $c_2 > 0$.

On the set $\bigcup A_i$, one has $\sum_{i=0}^{n-1} \tau_i \geq n \E(\tau) +
x$ by construction. To conclude, we should bound from below the measure of
this set. We have
\begin{equation*}
  \p\pare[\Big]{\bigcup A_i} \geq \sum_{i=0}^{n-1} \p(A_i) - \sum_{i\neq j=0}^{n-1} \p(A_i \cap A_j)
  \geq c_1 \frac{n}{x^p} - c_2 \frac{n^2}{x^{2p}}\,.
\end{equation*}
If $n$ is large enough, one has $c_2 n^2/x^{2p} \leq c_1 n/(2 x^p)$ when $x
\geq \sqrt{n}$. Therefore, we get $\p(\bigcup A_i)\geq (c_1/2) n/x^p$,
proving the desired result. As the estimate is trivial for bounded $n$, the
result follows.
\end{proof}

\begin{proof}[Proof of Proposition~\ref{prop:deviation_f_lower} for $p>2$]
Fix some $n\in \N$. Let $N$ denote the number of visits to $0$ of the Markov
chain $Y_i$ starting from $0$ strictly before time $n$. Then, given the
definition of $f$, one has
\begin{equation*}
  \sum_{i=0}^{n-1} f(Y_i) = n \pi\{0\} - N.
\end{equation*}
Therefore, for any $x \geq 0$,
\begin{equation*}
  \accol*{\sum_{i=0}^{n-1} f(Y_i) \geq x} = \{N \leq n \pi\{0\} - x\} = \accol*{\sum_{j=0}^{[n\pi\{0\}-x]-1} \tau_j \geq n}.
\end{equation*}
Let $m =[n\pi\{0\}-x]$. It is positive when $x\leq \kappa^{-1}n$, if $\kappa$
is large enough. We write $n$ as $m \E(\tau) + y$ for some $y$. As $\E(\tau)
= 1/\pi\{0\}$ by Kac formula, we have
\begin{equation*}
  y = n- [n\pi\{0\}-x]/\pi\{0\} \geq x/\pi\{0\}.
\end{equation*}
If $x\geq \kappa n^{1/p}$ with large enough $\kappa$, then $y\geq n^{1/p}$.
Hence, we can apply Lemma~\ref{lem:deviation_tau} to obtain
\begin{equation*}
  \p_0\pare*{\sum_{i=0}^{n-1} f(Y_i) \geq x} \geq C_1^{-1} \frac{m}{y^p} \geq C_2^{-1} \frac{n}{x^p}.
\end{equation*}
We obtain the same lower bound for the random walk started from $\pi$, with
an additional multiplicative factor $\pi\{0\}$.
\end{proof}

\begin{proof}[Proof of Proposition~\ref{prop:deviation_f_lower} for $p=2$]
In this case, $\sum_{j=0}^{n-1} \bar\tau_j / \sqrt{n\log n}$ converges to a
gaussian (see for instance~\cite{feller_2}). Following the proof of
Lemma~\ref{lem:deviation_tau}, one deduces first that this lemma holds
trivially for any $x \in [n^{1/p}, \sqrt{n \log n}]$, and also that it holds
for any $x\geq \sqrt{n\log n}$. It follows then from the same proof as in the
$p>2$ case that the proposition holds for all $x\in [\kappa n^{1/p},
\kappa^{-1}n]$.
\end{proof}

\begin{proof}[Proof of Proposition~\ref{prop:deviation_f_lower} for $p<2$]
In this case, $\sum_{j=0}^{n-1} \bar\tau_j / n^{1/p}$ converges to a stable
law (which is totally asymmetric of index $p$, see~\cite{feller_2}). Hence,
Lemma~\ref{lem:deviation_tau} holds for any $x\geq n^{1/p}$. It follows then
from the same proof as in the $p>2$ case that the proposition holds for all
$x\in [\kappa n^{1/p}, \kappa^{-1} n]$.
\end{proof}

\subsection{Young towers}

Consider now a Young tower $T: Z\to Z$ with invariant measure $\pi$ for which
the return time $\tau$ to the basis $Z_0$ of the tower satisfies
$\pi\{\tau=n\} \sim c/n^{p+1}$ on $Z_0$, for some $p>1$. In perfect analogy
with the previous paragraph, we define a function $f$ by $f = \pi(Z_0) -
\un_{Z_0}$. Its average under $\pi$ vanishes. The corresponding Markov chain
satisfies \H for the Hölder norm on the tower, see for
instance~\cite{gouezel_melbourne} and references therein.

Starting from $Y_0$ distributed according to $\pi$, we can consider $Y_0,
T(Y_0),\dotsc, T^{n-1}(Y_0)$, or the dual Markov chain $Y_0,\dotsc, Y_{n-1}$.
Then $Y_0,\dotsc, Y_{n-1}$ is distributed as $T^{n-1}(Y_0),\dotsc, Y_0$, as
explained at the beginning of Section~\ref{sec:Young_concentration}. It
follows that moderate deviations controls for one process or the other are
equivalent. We will state the lower bound statement for the Markov chain,
but we will prove it using the dynamical time direction.

\begin{proposition}
\label{prop:deviation_f_lower_Young} In this context, assume $p>2$. 
There exists
$\kappa>0$ such that, for any $n\in \N^*$ and any $x \in [\kappa n^{1/p},
\kappa^{-1} n]$,
\begin{equation*}
  \p\pare*{\sum_{i=0}^{n-1} f(Y_i) \geq x} \geq \kappa^{-1} \frac{n}{x^p}.
\end{equation*}
\end{proposition}
\begin{proof}
We work using the dynamical time direction. Starting from a point in the
basis $Z_0$ of the tower, let $\tau_0,\tau_1,\dotsc$ denote the lengths of
the successive excursions out of $Z_0$. The proof will be the same as for
Proposition~\ref{prop:deviation_f_lower} (notice that the statement is
exactly the same). The only difference is that the successive returns to the
basis are not independent, which means that the proof of
Lemma~\ref{lem:deviation_tau} has to be amended. We only give the proof for
$p>2$, as the other cases are virtually identical.

Let $T_0:Z_0 \to Z_0$ be the map induced by $T$ on the basis. It preserves
the probability $\pi_0$ induced by $\pi$ on $Z_0$. By definition, $T_0$ is a
Gibbs-Markov map with onto branches, i.e., there is a partition $\alpha_0$ of
$Z_0$ into positive measure subsets, such that $T_0$ maps bijectively each
$a\in \alpha_0$ to $Z_0$, with the following bounded distortion property. A
length $k$ cylinder is a set of the form $[a_0,\dotsc, a_{k-1}] =
\bigcap_{i<k} T_0^{-i} a_i$ for some $a_0,\dotsc, a_{k-1} \in \alpha_0$. Then
there exists a constant $C$ such that, for any $k>0$, for any length $k$
cylinder $A$ and for any measurable set $B$,
\begin{equation}
\label{eq:bounded_distortion}
  C^{-1} \pi_0(A) \pi_0(B) \leq \pi_0(A \cap T_0^{-k} B) \leq C \pi_0(A) \pi_0(B).
\end{equation}
(See for instance the last line in Section~1 of~\cite{aaronson_denker}.) This
estimate readily extends if $A$ is a union of length $k$ cylinders.

We can now prove the analogue of Lemma~\ref{lem:deviation_tau} in our
situation. Let $\bar\tau_i = \tau_i - \E(\tau_i)$. Define
\begin{equation*}
  A_i = \{\bar \tau_i \geq 3x \} \cap \accol[\Big]{\sum_{j=0}^{i-1} \bar\tau_j \leq x}
  \cap \accol[\Big]{\sum_{j=i+1}^{n-1} \bar \tau_j \leq x}
  =A_i^1 \cap A_i^2 \cap A_i^3\,.
\end{equation*}
We should show that, if $x \geq \sqrt{n}$, then $\pi_0(A_i) \geq c_1/x^p$ for
some $c_1>0$ independent of $i$ or $n$, and that $\pi_0(A_i \cap A_j) \leq
c_2/x^{2p}$ for $i<j$. Then, the proof of Lemma~\ref{lem:deviation_tau}
applies. In this lemma, the inequality $\p(A_i) \geq c_1/x^p$ follows from
independence and the fact that $\p(A_i^2) \geq c$ and $\p(A_i^3) \geq c$ and
$\p(A_i^1) \geq c/x^p$. In our context, these three inequalities still hold
(the first two ones follow from the fact that the Birkhoff sums of $\tau$
satisfy the central limit theorem or converge to a stable law,
see~\cite{aaronson_denker_central} and~\cite{aaronson_denker}, and the last
one from the assumptions on the tails of $\tau$), but independence fails. It
will be replaced by~\eqref{eq:bounded_distortion}. Let us give the details.
Recall that $\bar\tau_i = \tau(T_0^i z) - \pi (\tau) \coloneqq \bar\tau(T_0^i
z) $. Define
\[
B_1 = \{y \st \bar\tau(y) \geq 3x\}  \, , \,
B_2 = \accol[\Big]{y \st \sum_{j=0}^{i-1} \bar\tau(T_0^j y) \leq x}  \, \text{ and } \,
B_3 = \accol[\Big]{y \st \sum_{j=0}^{\mathclap{n-1-(i+1)}} \bar\tau(T_0^j y) \leq x} \, .
\]
We have $A_i^1 = T_0^{-i}(B_1)$,  $A_i^2 = B_2$ and $A_i^3 = T_0^{-(i+1)}
B_3$. Therefore,
\[
\pi_0(A_i) = \pi_0 \pare[\big]{B_2 \cap T_0^{-i}(B_1 \cap T_0^{-1}(B_3))} \, .
\]
Applying inequality~\eqref{eq:bounded_distortion} with $k=i$,  $A = B_2$ and
$B = B_1 \cap T_0^{-1}(B_3)$ (which is possible since $B_2$ is a union of
length $i$ cylinders since $\tau$ is constant on elements of $\alpha_0$), we
get $ \pi_0(A_i) \geq C^{-1} \pi_0(B_2) \pi_0(B_1 \cap T_0^{-1}(B_3))$. Next,
applying again~\eqref{eq:bounded_distortion} this time with $k=1$, $A = B_1$
and $B = B_3$ (which is possible since $ B_1$ is a union of length $1$
cylinders), we have $ \pi_0(B_1 \cap T_0^{-1}(B_3)) \geq C^{-1}  \pi_0(B_1)
\pi_0(B_3)$. So overall,
\[
\pi_0(A_i) \geq C^{-2} \pi_0(A_i^2) \pi_0(A_i^1) \pi_0(A_i^3) \, .
\]
This inequality replaces the independence assumption and implies
that~$\pi_0(A_i) \geq c_1/x^p$.  The inequality $\pi_0(A_i \cap A_j) \leq
c_2/x^{2p}$ is proved in the same way, using the upper bound
in~\eqref{eq:bounded_distortion}.
\end{proof}

\subsection{Harris Markov chains with state space \texorpdfstring{$[0,1]$}{[0,1]}}
Let $a = p-1$ with $p >1 $. Let $\lambda$ denote the Lebesgue measure on
$[0,1]$. Define the probability laws $\nu$ and $\pi$ by
\[
\nu = (1+a) x^a \lambda \, \text { and } \, \pi = a x^{a-1} \lambda \, .
\]
We define now a strictly stationary Markov chain by specifying its transition
probabilities $K(x,A)$ as follows:
\[
K(x,A)=(1-x)\delta_{x}(A)+x\nu(A)\, ,
\]
where $\delta_{x}$ denotes the Dirac measure. Then $\pi$ is the unique
invariant probability measure of the chain with transition probabilities
$K(x, \cdot)$. Let $(Y_i)_{i \in {\Z}}$ be the stationary Markov chain on
$[0,1]$ with transition probabilities $K(x, \cdot)$ and law $\pi$. For
$\gamma>0$, we set
\[
c_{a, \gamma} = \frac{a}{a+ \gamma} \, , \, X_i = f_{\gamma} (Y_i) - \E(f_{\gamma} (Y_i)) \coloneqq Y_i^{\gamma} - c_{a, \gamma} \, \text{ and } \, S_n = \sum_{i=0}^{n-1} X_i \, .
\]
Denote by
\[
\beta_n \coloneqq \frac 12
 \pi \pare[\Big]{\sup_{\norm{f}_{\infty} \leq 1} \abs{K^n(f) - \pi (f)}} \, ,
\]
and set $T(x) =1-x$. According to Lemma 2 in Doukhan, Massart and Rio (1994),
\[
\beta_n \leq 3 \, \E_{\pi} (T^{[n/2]}) \, .
\]
Note now that for any $b > -1$,
\[
 \int_0^1 (1-x)^k x^b \dd x = k^{-(b+1)} \int_0^{k} (1-x/k)^k x^b \dd x \, .
\]
Since for any $x \in [0,1]$, $\log (1-x) \leq -x$, it follows that
\begin{equation} \label{majorationevidentegamme}
 \int_0^1 (1-x)^k x^b \dd x \leq k^{-(b +1)} \int_0^{k} e^{-x} x^b \dd x \leq k^{-(b +1)} \Gamma (b+1) \, ,
\end{equation}
implying that
\[
 \E_{\pi} (T^k) \leq a \Gamma (a) k^{-a} \, .
\]
Therefore
\begin{equation*}
 \sup_{\norm{f}_{\infty} \leq 1} \pi \pare[\big]{\abs{K^n(f) - \pi (f)}} \leq  2 \beta_n \leq C n^{-a} \, ,
\end{equation*}
which shows that the condition \Hun is satisfied for the two norms
$\norm{f}_{\infty}$ and $\norm{f}_{BV} $. For the norm $\norm{f}_{\infty}$,
the condition \Hdeux is trivially satisfied with $C_2=1$. Hence,
Theorem~\ref{ModDevMC} applies to $(f_{\gamma} (Y_i))_{i \in \Z}$. We shall
verify that the condition \Hdeux also holds for the norm $\norm{f}_{BV}  =
\norm{f}_{\infty} + \abs{\dd f}$ at the end of this section. Concerning the
lower bound, the following proposition holds:

\begin{proposition}
\label{prop:deviation_f_lower_Harris} Let $(Y_i)_{i \in \N}$ be a stationary
Markov chain with transition kernel described above. Assume $p>1$ and $\gamma
>0$. There exists $\kappa>0$ such that, for any $n\in \N^*$ and any $x \in
[\kappa n^{1/p}, \kappa^{-1} n]$,
\begin{equation*}
  \p\pare*{\max_{1 \leq k \leq n} \abs[\Big]{\sum_{i=0}^{k-1} (Y^{\gamma}_i - \E (Y^{\gamma}_i  )  )}  \geq x} \geq \kappa^{-1} \frac{n}{x^p} \, .
\end{equation*}
\end{proposition}
\begin{proof} We first define a sequence $(T_k)_{k \geq 0}$
of stopping time as follows:
\[
 T_0 = \inf\{i > 0 \, : \, Y_i \neq Y_{i-1}\} \, \text{ and } \,
 T_k = \inf\{i > T_{k-1} \, : \, Y_i \neq Y_{i-1}\} \, \text{ for $k > 0$} \, .
\]
Let $\tau_k = T_{k+1} - T_{k}$. The r.v.'s $(Y_{T_k} , \tau_k )_{k \geq 0}$
are i.i.d., $Y_{T_k}$ has law $\nu$ and the conditional distribution of
$\tau_k $ given $Y_{T_k}=y$ is the geometric distribution $\calG (1-y)$. We
have in particular that $\tau_0$ is integrable.  The key inequality for
proving the lower bound is the following one:
\begin{equation} \label{bornesuppresquefinale}
\p\pare[\Big]{\max_{0 \leq k \leq n-1} \tau_k \abs{X_{T_k}} \geq 24 x }  \leq 9 \p\pare[\Big]{\max_{1 \leq k \leq [n\E (\tau_1) ] +1} \abs{S_{k}} \geq  x } + 3 \p(T_n \geq 2 [n\E (\tau_1) ] +1 ) \, .
\end{equation}
Before proving it, let us show how it will entail the lower bound.

Using the fact the r.v.'s $(Y_{T_k} , \tau_k )_{k \geq 0}$ are i.i.d.,
$Y_{T_k}$ has law $\nu$ and the conditional distribution of $\tau_k $ given
$Y_{T_k}=y$ is the geometric distribution $\calG (1-y)$, straightforward
computations imply that for $x \geq \kappa n^{1/p} $ with $\kappa$ large
enough,
\begin{equation} \label{b1max} \p\pare[\big]{\max_{0 \leq k \leq
n-1} \tau_k \abs{X_{T_k}} \geq  24 x } \geq  C_{p, \gamma} \frac{n}{x^p} \, ,
\end{equation}
where
\[
C_{p, \gamma} = \frac{1}{4}  \pare[\Big]{\frac{c_{a, \gamma} \eta}{ 48}}^{p} p \Gamma (p)  \, , \, \text{ with $\eta = 1 - (c_{a, \gamma} /2)^{1/\gamma} $.}
\]
On the other hand,
\[
\p(T_n \geq 2 [n\E (\tau_1) ] +1 ) \leq \p\pare[\Big]{T_0 + \sum_{i=0}^{n-1} (\tau_i - \E(\tau_i)) \geq [n\E (\tau_1) ]}\,.
\]
Since $\E(\tau_1) \geq 1$, this gives
\[
\p(T_n \geq 2 [n\E (\tau_1) ] +1 ) \leq \p\pare[\big]{T_0 \geq n/2} +\p\pare[\Big]{\sum_{i=0}^{n-1} (\tau_i - \E(\tau_i)) \geq n/2} \, .
\]
Since $\p\pare[\big]{T_0 \geq n/2} \leq \int_0^1 (1-x)^{n/2} d\pi(x) $,
according to~\eqref{majorationevidentegamme}
\[
\p\pare[\big]{T_0 \geq n/2} \leq 2^a a n^{-a} \Gamma (a) \, .
\]
Assume from now that $p \geq 2$. Since the $(\tau_k)_{k \geq 0}$ are i.i.d.,
the Fuk-Nagaev inequality for independent random variables (see for instance
Theorem B.3 in Rio (2000) and its proof) gives that, for any $u > 0$ and any
$v_n^2 (u) \geq \sum_{i=0}^{n-1} \E ((\tau_i \wedge u)^2)$,
\begin{equation} \label{FNineqind} 
  \p\pare[\Big]{\sum_{i=0}^{n-1} (\tau_i - \E(\tau_i)) \geq n/2} \leq n \p(\tau_1 \geq u )
 + \exp \pare[\Big]{- \frac{n}{4u} \log \pare[\big]{1 + \frac{n u}{2 v_n^2 (u)}}} \, .
\end{equation}
We shall apply this inequality with the following choice of $u$:
\begin{equation*}
u = \frac{n}{8 (p-1)} \, .
\end{equation*}
The selection of $v_n^2(u)$ will be different if $p>2$ or if $p=2$. Assume
first that $p > 2$. In this case, we take $v_n^2 (u) = n \E (\tau_1^2) $.
Since $Y_{T_k}$ has law $\nu$ and the conditional distribution of $\tau_k $
given $Y_{T_k}=y$ is the geometric distribution $\calG (1-y)$, simple
computations give
\[
\E (\tau^2_1) = \frac{p^2}{(p-1)(p-2)} : =c_p \text{ and then } v_n^2 (u)= c_p n \, .
\]
On another hand, if $p=2$, we first note that
\[
 \E((\tau_1 \wedge u)^2) = \E(\tau^2_1 {\bf 1}_{\tau_1 \leq u}) + u^2 \p(\tau_1 \geq u )
 \leq \sum_{\ell = 0}^{[u-1]} (2 \ell +1) \p(\tau_1 \geq \ell) + u^2 \p(\tau_1 \geq u ) \, .
\]
Now~\eqref{majorationevidentegamme} implies that $\p(\tau_1 \geq \ell) \leq 2
\ell^{-2}$ . Therefore, if $n \geq 8$,
\[
 \E ((\tau_1 \wedge u)^2) \leq \log (u ) + 5 \leq 3 \log (n ) \, .
\]
So, in case $p=2$, we take $v_n^2 (u) = 3 n \log n $.

If $p>2$, then~\eqref{FNineqind} together with the fact that,
by~\eqref{majorationevidentegamme}, $ \p(\tau_1 \geq \ell) \leq p \Gamma(p)
\ell^{-p}$ imply that
\begin{equation} \label{FNineqind2}
\p\pare[\Big]{\sum_{i=0}^{n-1} (\tau_i - \E(\tau_i)) \geq n/2} \leq p
\Gamma (p) \times (8 (p-1))^{p} n^{-p+1} + (16 (p-1) c_p )^{2 (p-1)} n^{-2(p-1)} \, .
\end{equation}
So, overall, starting from~\eqref{bornesuppresquefinale} and taking
account~\eqref{b1max} and~\eqref{FNineqind2}, we get that for $\kappa$ large
enough
\begin{multline*}
 \p\pare[\Big]{\max_{1 \leq k \leq [n\E (\tau_1) ] +1} \abs{S_{k}} \geq x}  \geq 9^{-1} p\Gamma (p) \Big \{4^{-1} \pare[\Big]{\frac{c_{a, \gamma} \eta}{48}}^{p} n x^{-p}- 6 (8 (p-1) )^{p} n^{-p+1} \Big \}  \\
 - 3^{-1}(16 (p-1) c_p )^{2 (p-1)} n^{-2(p-1)} \, .
\end{multline*}
Since $n^{-p} \leq (x \kappa)^{-p}$ and $\E(\tau_1) = \frac{p}{p-1} \leq 2$,
it follows that for $\kappa$ large enough
\[
 \p\pare[\Big]{\max_{1 \leq k \leq  2n +1} \abs{S_{k}} \geq x}  \geq \frac{2n}{\kappa x^p} \, ,
\]
giving the lower bound when $p>2$.

We turn now to the case when $p=2$, we derive this time
\begin{equation*}
\p\pare[\Big]{\sum_{i=0}^{n-1} (\tau_i - \E(\tau_i)) \geq n/2} \leq 2
\times 8^{2} n^{-1} + (3 \times 16 )^{2} (\log n )^2 n^{-2} \, .
\end{equation*}
Proceeding as before, the lower bound follows.

We end the proof by considering the case $1 < p <2$. Let $u$ be a positive
real and set ${\bar \tau}_i = (\tau_i \wedge u)$. Note that
\begin{multline*}
 \sum_{i=0}^{n-1} (\tau_i - \E(\tau_i)) = \sum_{i=0}^{n-1} ({\bar \tau}_i - \E({\bar \tau}_i)) +\sum_{i=0}^{n-1} ((\tau_i -u)_+ - \E((\tau_i -u)_+)) \\
 \leq \sum_{i=0}^{n-1} ({\bar \tau}_i - \E({\bar \tau}_i)) +\sum_{i=0}^{n-1} (\tau_i -u)_+ \, ,
\end{multline*}
which implies that
\[
\p\pare[\big]{\sum_{i=0}^{n-1} (\tau_i - \E(\tau_i)) \geq n/2} \leq \p\pare[\big]{({\bar \tau}_i - \E({\bar \tau}_i)) \geq n/2} + \sum_{i=0}^{n-1} \p(\tau_i \geq u) \, .
\]
Next, by Markov inequality, we get that for any $u > 0$,
\begin{multline*}
\p\pare[\big]{\sum_{i=0}^{n-1} (\tau_i - \E(\tau_i)) \geq n/2} \leq 4 n^{-1} \E ((\tau_1 \wedge u)^2) +n\p(\tau_1 \geq u) \\
\leq 4 n^{-1} \E (\tau^2_1 {\bf 1}_{\tau_1 \leq u}) + 4 n^{-1} u^2 \p(\tau_1 \geq u) + n \p(\tau_1 \geq u) \, .
\end{multline*}
We have
\[
  \E (\tau^2_1 {\bf 1}_{\tau_1 \leq u}) \leq 2 \int_0^u t \p(\tau_1 \geq t) \dd t
  \leq 1 + 2 p \Gamma (p) \int_1^u \frac{t}{[t]^p} \dd t
  \leq 1 + \frac{2^{p+1} p \Gamma (p)}{2-p} u^{2-p} \, .
\]
Therefore, choosing $u = n$, we get overall that, in the case $1 < p <2$,
\[
\p\pare[\big]{\sum_{i=0}^{n-1} (\tau_i - \E(\tau_i)) \geq n/2} \leq 4 n^{-1} + n^{1-p} p \Gamma (p) \pare[\Big]{5 + \frac{2^{p+3}}{2-p}} \, .
\]
Proceeding as before, the lower bound follows.

\smallskip To end the proof of the lower
bound, it remains to prove inequality~\eqref{bornesuppresquefinale}. With
this aim, setting
 \[
Z_0= T_0 X_0 \ \text{ and } Z_k = \tau_{k-1} X_{T_{k-1}} \text{ for $k \geq 1$}
\]
we note that
\[
 \max_{0 \leq k \leq n-1} \tau_k \abs{X_{T_k}} \leq \max_{0 \leq k \leq n} \abs{Z_k} \, .
\]
But for any $k \geq 1$, $Z_k = \sum_{i=0}^{k} Z_i - \sum_{i=0}^{k-1} Z_i $.
Therefore
\[
 \max_{0 \leq k \leq n} \abs{Z_k} \leq 2 \max_{0 \leq k \leq n} \abs[\Big]{\sum_{i=0}^{k} Z_i} \, .
\]
The above considerations imply that
\[
\p\pare[\Big]{\max_{0 \leq k \leq n-1} \tau_k \abs{X_{T_k}} \geq24 x } \leq \p \pare[\Big]{\max_{0 \leq k \leq n} \abs[\Big]{\sum_{i=0}^{k} Z_i}
 \geq 12 x } \, . \]
$(Z_k)_{k \geq 0}$ being a sequence of independent random variables,
Etemadi's inequality entails that
\[
 \p \pare[\Big]{\max_{0 \leq k \leq n} \abs[\Big]{\sum_{i=0}^{k} Z_i}
 \geq 12x} \leq 3 \p \pare[\Big]{\abs[\Big]{\sum_{i=0}^{n} Z_i}
 \geq  4x } \, .
\]
Note now that
\begin{align*}
 \sum_{i=0}^{n} Z_i = \sum_{k=0}^{T_0-1} X_0 + \sum_{i=1}^{n} (T_{i} - T_{i-1} ) X_{T_{i-1}} = \sum_{k=0}^{T_0-1} X_k + \sum_{i=1}^{n} \sum_{j=T_{i-1}}^{T_{i} -1} X_{j} = \sum_{k=0}^{T_{n} -1} X_k \, .
\end{align*}
Therefore
\[
 \p \pare[\Big]{\abs[\Big]{\sum_{i=0}^{n} Z_i}
 \geq 4 x }
  \leq \p\pare[\Big]{\abs[\Big]{\sum_{i=0}^{[n \E(\tau_1)] -1} X_{i}} \geq 2 x }
  + \p\pare[\Big]{\abs{S_{T_{n}} - S_{[n \E(\tau_1)]}} \geq 2 x } \, .
\]
Inequality~\eqref{bornesuppresquefinale} follows from all the considerations
above, together with the fact that
\[
  \p\pare[\Big]{\abs{S_{T_{n}} - S_{[n \E(\tau_1)]}} \geq  2 x }
  \leq 2 \p\pare[\Big]{\max_{1 \leq k \leq [n\E (\tau_1) ] +1} \abs{S_{k}} \geq  x }
    + \p(T_n \geq 2 [n\E (\tau_1) ] +1 ) \, .
  \qedhere
\]
\end{proof}

To complete this section, it remains to show that the transition operator $K$
of the Markov chain satisfies condition \Hdeux for the semi norm $\abs{\dd
f}$. With this aim, we first note that
\[
K(f)(x) = (1-x) f(x) + x \nu (f) \, .
\]
So iterating, we get for any positive integer $n$,
\[
K^n f(x) = (1-x)^n f(x) + \sum_{k=0}^{n-1} x (1-x)^k \nu (K^{n-1-k} (f)) \, .
\]
Therefore, we infer that
\begin{multline*}
  K^n f(x) = (1-x)^n (f(x) - \nu(f) ) + \nu (K^{n-1} (f)) \\
  + \sum_{k=1}^{n-1} (1-x)^{n-k} \pare[\big]{\nu (K^{k-1}(f) ) -\nu (K^{k} (f) )} \, .
\end{multline*}
It follows that
\begin{equation}\label{decomp}
\abs{\dd K^n (f)} \leq 3 \abs{\dd f} + \sum_{k=1}^{n-1} \abs{\nu (K^{k-1} (f) ) -\nu (K^{k} (f) )} \, .
\end{equation}
Setting $g_0 = f - f(0)$, note now that, for any positive integer $k$,
\[
 \abs{\nu (K^{k-1} (f) ) -\nu (K^{k} (f) )} = \abs{\nu (K^{k-1} (g_0) ) - \nu (K^{k}(g_0) )} \, .
\]
Therefore
\[
 \sum_{k=1}^{n-1} \abs{\nu (K^{k-1} (f) ) -\nu (K^{k} (f) )}
 \leq \sum_{k=1}^{n-1} \int_0^1 \abs{K^{k-1} (g_0) (x)- K^{k} (g_0) (x)} \dd\nu (x) \, .
\]
But $ \sup_{x \in [0, 1]} \abs{g_0 (x)} \leq \abs{\dd f}$. Hence
\begin{equation}\label{decomp2}
\begin{split}
  \int_0^1 \abs{K^{k-1} (g_0) (x)- K^{k} (g_0) (x)} \dd\nu (x) & = \int_0^1 \abs{(\delta_x K^{k-1} - \delta_x K^{k} ) (g_0)} \dd\nu (x)
  \\ & \leq \abs{\dd f} \int_0^1 \abs{\delta_x K^{k-1} - \delta_x K^{k}} \dd\nu (x) \, .
\end{split}
\end{equation}
From~\eqref{decomp} and~\eqref{decomp2}, to complete the proof of the fact
that $K$ satisfies \Hdeux, it remains to show that
\begin{equation} \label{mainaimH2K}
 \sum_{k\geq 1} \int_0^1 \abs{\delta_x K^{k-1} - \delta_x K^{k}} \dd\nu (x) < \infty \, .
\end{equation}
Set $T(x) = 1-x$. According to the computations leading to the first
inequality on page 76 of~\cite{Doukhan_Massart_Rio}, we have, for any integer
$k \geq 2$
\[
\abs{\delta_x K^{k-1} - \delta_x K^{k}} \leq 2 (T(x))^{k-1} + \sum_{i=1}^{k-1} (1-T(x) ) (T(x))^{i-1} \abs{\nu K^{k-1-i} - \nu K^{k-i}} \, ,
\]
implying that
\begin{equation} \label{1rstboundvariation}
 \int_0^1 \abs{\delta_x K^{k-1} - \delta_x K^{k}} \dd\nu (x) \leq 2 \E_{\nu}(T^{k-1}) + \sum_{i=1}^{k-1} \E_{\nu} ((1-T) T^{i-1} ) \abs{\nu K^{k-1-i} - \nu K^{k-i}} \, .
\end{equation}
But, by taking into account~\eqref{majorationevidentegamme}, we get
\begin{equation}
\label{1rstboundvariation2}
 \E_{\nu}(T^k) = (1+a) \int_0^1 (1-x)^k x^a \dd x \leq k^{-(a+1)} (a+1) \Gamma (a+1)
\end{equation}
and, for any integer $i \geq 2$,
\begin{equation} 
  \label{1rstboundvariation2bis}
  \E_{\nu} ((1-T) T^{i-1} ) = (1+a) \int_0^1 (1-x)^{i-1} x^{a+1} \dd x \leq
  (i-1)^{- (a +2)} (a+1) \Gamma (a+2) \, .
\end{equation}
We need now to give an upper bound of $ \abs{\nu K^{j} - \nu K^{j+1}} $ for
any non negative integer $j$. With this aim, we first notice that
\[
 K^{j} (f) - K^{j+1} (f) = s K^{j} (f) - s \nu (K^{j} (f) )
\]
where $s(x) = x$. Therefore setting $ \displaystyle \mu = \frac{s(x)}{\nu(s)}
\nu$, we have
\[
 \nu K^{j} - \nu K^{j+1} = \nu (s) \pare[\big]{\mu K^{j} - \nu K^{j}} \, .
\]
Taking into account the relation~(9.11) in~\cite{rio_book}, this gives
\begin{equation} \label{1rstboundvariation3}
(\nu (s) )^{-1} (\nu K^{j} - \nu K^{j+1} ) =
\sum_{\ell=1}^j a_{\ell} \nu Q^{j-\ell} + \mu Q^j - \nu Q^j \, .
\end{equation}
where
\[
Q (x,A) = K(x,A) - s (x) \nu (A) = T (x) \delta_x (A) \ \text{ and } \ a_{\ell} = \mu K^{\ell-1} (s) - \nu K^{\ell -1} (s) \, .
\]
If we can prove that for any positive integer $\ell$, $a_\ell$ is non
negative, the relation~\eqref{1rstboundvariation3} will imply that the
signed measures $ \nu K^{j} - \nu K^{j+1} $ of null mass can be rewritten as
the differences of two positive measures with finite mass (the second one
being equal to $ \nu Q^j $), and therefore we will have
\begin{equation} \label{1rstboundvariation4}
 \abs{\nu K^{j} - \nu K^{j+1}} \leq 2 \nu (s) \nu Q^j ({\bf 1}) = 2 \nu (s) \E_{\nu} (T^j) \, .
\end{equation}
Hence, starting from~\eqref{1rstboundvariation} and taking into
account~\eqref{1rstboundvariation2},~\eqref{1rstboundvariation2bis}
and~\eqref{1rstboundvariation4}, we will get that for any integer $k \geq 2$,
\begin{equation} \label{1rstboundvariation5}
 \int_0^1 \abs{\delta_x K^{k-1} - \delta_x K^{k}} \dd\nu (x)
 \leq C_a \pare[\Big]{\frac{1}{k^{a+1}} + \sum_{i=1}^{k-1} \frac{1}{i^{a +2}} \times \frac{1}{(k-i)^{a+1}}}
 \leq {\tilde C}_a \frac{1}{k^{a+1}} \, ,
\end{equation}
provided that one can prove that, for any positive integer $\ell$, $a_\ell$
is non negative. This can been proved by using~\eqref{1rstboundvariation3}
and  the arguments developed in the proof of Lemma~9.3 in~\cite{rio_book}. We
complete the proof by noticing that~\eqref{1rstboundvariation5}
implies~\eqref{mainaimH2K} since $a > 0$.

\bibliography{biblio}
\bibliographystyle{amsalpha}

\end{document}